\documentclass[a4paper,reqno]{amsart}
\usepackage{cite}
\usepackage{amssymb}
\usepackage[vcentermath]{youngtab}
\usepackage[nohug]{diagrams}

\def\tab(#1){\,\mbox{\tiny$\young(#1)$}\,}

\title{On homomorphisms indexed by semistandard tableaux}

\author[S.~Lyle]{Sin\'ead Lyle}
\address{School of Mathematics, University of East Anglia, Norwich NR4 7TJ, UK.}
\email{s.lyle@uea.ac.uk}
\thanks{The author acknowledges support from the EPSRC}

\numberwithin{equation}{section}
\numberwithin{figure}{section}
\newtheorem{lemma}{Lemma}[section]
\newtheorem{theorem}[lemma]{Theorem}
\newtheorem{proposition}[lemma]{Proposition}
\newtheorem{corollary}[lemma]{Corollary}

\newtheorem*{theorem*}{Theorem}
\newtheorem*{proposition*}{Proposition}
\newtheorem*{lemma*}{Lemma}
\newtheorem*{definition*}{Definition}
\newtheorem*{caution*}{Caution}

\theoremstyle{definition}

\theoremstyle{remark}
\newtheorem*{remark}{Remark}
\newtheorem{ex}{Example}

\newcommand{\h}{\mathcal{H}}
\newcommand{\hla}{\mathcal{H}^{\rhd \la}}
\newcommand{\hmu}{\mathcal{H}^{\rhd \mu}}

\newcommand{\symn}{\mathfrak{S}_n} 
\newcommand{\sym}{\mathfrak{S}} 

\newcommand{\C}{\mathbb{C}}

\newcommand{\Z}{\mathbb{Z}}

\newcommand{\la}{\lambda}

\newcommand{\ft}{\mathfrak{t}}
\newcommand{\fs}{\mathfrak{s}}
\newcommand{\fu}{\mathfrak{u}}
\newcommand{\fv}{\mathfrak{v}}

\newcommand{\q}{\hat{q}}
\newcommand{\gauss}[2]{{{#1} \brack {#2}}}
\newcommand{\seq}[3]{\langle{#1},{#2},{#3}\rangle}
\newcommand{\rep}[2]{\begin{array}{l}{#1}\\{#2}\end{array}}
\DeclareMathOperator{\CC}{C}
\DeclareMathOperator{\D}{D}
\newcommand{\Dt}{\D^\flat}
\newcommand{\Per}{\pi}
\newcommand{\calZ}{\mathcal{Z}}

\DeclareMathOperator{\Std}{Std}
\DeclareMathOperator{\rowstd}{RStd}

\DeclareMathOperator{\Hom}{Hom}
\DeclareMathOperator{\EHom}{EHom}

\DeclareMathOperator{\first}{first}
\DeclareMathOperator{\corank}{corank}
\DeclareMathOperator{\rank}{rank}
\DeclareMathOperator{\rrr}{r}
\newcommand{\RowT}{\mathcal{T}_{\rrr}}
\def\S{\mathsf{S}}
\def\T{\mathsf{T}}
\def\U{\mathsf{U}}
\def\V{\mathsf{V}}

\begin{document}
\begin{abstract}
We study the homomorphism spaces between Specht modules for the Hecke algebras $\h$ of type $A$.  We prove a cellular analogue of the kernel intersection theorem and a $q$-analogue of a theorem of Fayers and Martin and apply these results to give an algorithm which computes the homomorphism spaces $\Hom_\h(S^\mu,S^\la)$ for certain pairs of partitions $\la$ and $\mu$.  We give an explicit description of the homomorphism spaces $\Hom_\h(S^\mu,S^\la)$ where $\h$ is an algebra over the complex numbers, $\la=(\la_1,\la_2)$ and $\mu$ is an arbitrary partition with $\mu_1 \geq \la_2$.   

\end{abstract}
\maketitle

\section{Introduction}
The Hecke algebras $\h=\h_{F,q}(\sym_n)$ of the symmetric groups are classical objects of study and the most important open problem in their representation theory is to determine the structure of the Specht modules $S^\la$ where $\la$ is a partition of $n$.   In this area there are many obvious questions that remain unanswered.  For example, we rarely know the composition factors of $S^\la$ or their multiplicities. However, information on the Specht modules may be obtained by computing $\Hom_{\h}(S^\mu,S^\la)$ for $\mu$ and $\la$ partitions of $n$. 
An approach to this problem using the kernel intersection theorem was suggested by James, who gave an easy classification of $\Hom_{F\sym_n}(S^\mu, S^{(n)})$~\cite[Theorem 24.4]{James}.  This approach has subsequently been developed.  In particular, results of Fayers and Martin~\cite{FayersMartin:homs} have given us techniques to compute $\Hom_{F\sym_n}(S^\mu, S^{\la})$ for more general $\la$ (which they used in the same paper to give an elementary proof of the Carter-Payne theorem).  In this paper, we extend the most useful of their results to the Hecke algebra $\h$.  This enables us to give an algorithm, easily implemented on a computer, which will compute certain homomorphism spaces. Using this method we completely classify the homomorphism space $\Hom_{\h}(S^\mu,S^\la)$ where $\la$ has at most two parts, $\mu_1 \geq \la_2$ and $\h$ is defined over a field of characteristic zero.  

The paper is organised as follows.  We begin with the definition of the Hecke algebras and some background discussion.  We then state our main results and give some examples and applications.  The proofs of the main results, Theorem~\ref{hdtthm}, Theorem~\ref{Lemma7} and Propositions~\ref{first} to~\ref{last}  are deferred to the next section; in fact, we give only an indication of the proof of the last propositions, for reasons we discuss in Section~\ref{Spaces}.  We end with a brief discussion about homomorphisms between the Specht modules of Dipper and James.  
\section{Main results}

\subsection{The Hecke algebras of type $A$}
The definitions in this section are standard and may all be found the the book of Mathas~\cite{M:ULect}.  

For each integer $n\geq 0$, let $\sym_n$ be the symmetric group on $n$ letters.  If $R$ is a ring and $q$ an invertible element of $R$ then the Hecke algebra
$\h=\h_{R,q}(\sym_n)$ is defined to be the unital associative 
$R$-algebra with generators
$T_1,\dots,T_{n-1}$ subject to the relations
\begin{align*}
T_i T_j & = T_j T_i, && 1\leq i < j-1\leq n-2, \\
T_iT_{i+1}T_i &= T_{i+1}T_i T_{i+1}, && 1 \leq i \leq n-2, \\
(T_i+1)(T_i-q)&=0, &&1 \leq i \leq n-1,
\end{align*}
so that if $q=1$ then $\h \cong R\symn$.  If $w \in \sym_n$ can be written as 
$w=(i_1,i_1+1)\ldots(i_k,i_k+1)$ where $k$ is minimal, we define $T_w = T_{i_1} \ldots T_{i_k}$.  
Then $\h$ is a free $R$-module with a basis $\{T_w \mid w \in \sym_n\}$. 
An expression $w=(i_1,i_1+1)\ldots(i_k,i_k+1)$ with $k$ minimal is known as a reduced expression for $w$ and we define the length of $w$ by $\ell(w)=k$.  

Recall that a composition of $n\geq 0$ is a sequence $\mu=(\mu_1,\mu_2,\ldots,\mu_l)$ of non-negative integers that sum to $n$ and a partition is a composition with the additional property that $\mu_1 \geq \mu_2 \geq \ldots \geq \mu_l$.  If $\mu$ is a partition of $n$, write $\mu \vdash n$.  
We define a partial order $\unrhd$ on the set of compositions of $n$ by saying that $\la \unrhd \mu$ if 
\[ \sum_{i=1}^j \la_i \geq \sum_{i=1}^j \mu_i\]
for all $j$.  If $\la \unrhd \mu$ and $\la \neq \mu$, write $\la \rhd \mu$.  

Let $\mu$ be a composition of $n$. Define the corresponding Young diagram $[\mu]$ by
\[[\mu] = \{(r,c) \mid 1 \leq c \leq \mu_r \}.\]
A $\mu$-tableau $\T$ is a map $\T: [\mu] \rightarrow \{1,2,\ldots\}$; we think of this as a way of replacing the nodes $(r,c)$ of $[\mu]$ with the integers $1,2,\ldots$ and so may talk about the rows and columns of $\T$.  (Note that we use the English convention for writing our diagrams.) Let $\mathcal{T}(\mu)$ be the set of $\mu$-tableaux such that each integer $1,2,\ldots,n$ appears exactly once and let $\ft^\mu \in \mathcal{T}(\mu)$ be the tableau with $\{1,2,\ldots,n\}$ entered in order along the rows of $[\mu]$ from left to right and top to bottom.  If $\ft \in \mathcal{T}(\mu)$ and $I \subseteq \{1,2,\ldots,n\}$, say that $I$ is in row-order in $\ft$ if for all $i,j \in I$ with $i<j$ either $j$ lies in a lower row than $i$ (that is, a row of higher index); or $i$ and $j$ lie in the same row, with $j$ to the right of $i$.  Then $\ft^\mu$ is the unique tableau in $\mathcal{T}(\mu)$ with $\{1,2,\ldots,n\}$ in row-order.  

The symmetric group $\sym_n$ acts on the right on the elements of $\mathcal{T}(\mu)$ by permuting the entries in each tableau.  If $\ft \in \mathcal{T}(\mu)$, let $d(\ft)$ be the permutation such that $\ft = \ft^\mu d(\ft)$.  Let $\sym_{\mu}$ denote the row-stabilizer of $\ft^\mu$, that is, the set of all permutations $w$ such that each $i \in \{1,2,\ldots,n\}$ lies in the same row of $\ft^\mu$ as $\ft^\mu w$.   
We say that $\ft \in \mathcal{T}(\mu)$ is row-standard if the entries increase along the rows and standard if $\mu$ is a partition and the entries increase both along the rows and down the columns. Let $\rowstd(\mu) \subseteq \mathcal{T}(\mu)$ denote the set of row-standard $\mu$-tableaux and, if $\mu$ is a partition, let $\Std(\mu) \subseteq \rowstd(\mu)$ denote the set of standard $\mu$-tableaux.
Define    
\[m_\mu = \sum_{w \in \sym_\mu} T_w,\] and set $M^\mu$ to be the right $\h$-module
\[M^\mu = m_\mu \h.\]
Define $\ast:\h \rightarrow \h$ to be the anti-isomorphism determined by $T_{i}^\ast = T_{i}$
and if $\fs,\ft \in \rowstd(\mu)$ define 
\[m_{\fs\ft}=T_{d(\fs)}^\ast m_\mu T_{d(\ft)}.\]  
Then
\[\{m_{\fs\ft} \mid \fs,\ft \in \Std(\la) \text{ for some } \la \vdash n\}\]
is a cellular basis of $\h$ with respect to the partial order $\unrhd$ and the anti-isomorphism $\ast$.  In accordance with the theory of cellular algebras, if $\la \vdash n$ we define $\hla$ to be the free $R$-module with basis 
\[\{m_{\fs\ft} \mid \fs,\ft \in \Std(\nu) \text{ for some } \nu \vdash n \text{ such that } \nu \rhd \la\};\]
then $\hla$ is a two-sided ideal of $\h$.  Following Graham and Lehrer~\cite{GL}, we define the cell module $S^\la$, also known as a Specht module, to be the right $\h$-module 
\[S^\la = (\hla + m_\la)\h\]
and define $\pi_\la:M^\la \rightarrow S^\la$ to be the natural projection determined by $\pi_\la(m_\la)=\hla + m_\la$.  
These Specht modules are the main objects of interest in the study of the representation theory of the Hecke algebras $\h$ and the symmetric groups $\sym_n$.  One of the most important open problems in representation theory is to determine the decomposition matrices for the Hecke algebra $\h$, that is, compute the composition factors of the Specht modules.  In this paper, we study a closely related problem.  We consider homomorphisms between Specht modules $S^\mu$ and $S^\la$, for $\mu$ and $\la$ partitions of $n$.  

\subsection{Homomorphisms between Specht modules} \label{HomSection}
Suppose $\la$ is a partition of $n$ and let $\T$ be a $\la$-tableaux.  Say that $\T$ is of type $\mu$ if $\mu$ is the composition such that each integer $i \geq 1$ appears $\mu_i$ times in $\T$.    
Let $\mathcal{T}(\la,\mu)$ denote the set of $\la$-tableaux of type $\mu$.  We say that $\S \in \mathcal{T}(\la,\mu)$ is row-standard if the entries are non-decreasing along the rows and is semistandard if it is row-standard and the entries are strictly increasing down the columns.  Let $\mathcal{T}_{\text{r}}(\la,\mu)\subseteq \mathcal{T}(\la,\mu)$ denote the set of row-standard $\la$-tableaux of type $\mu$ and $\mathcal{T}_0(\la,\mu) \subseteq \mathcal{T}_{\text{r}}(\la,\mu)$ denote the set of semistandard $\la$-tableaux of type $\mu$. 
If $\fs \in \mathcal{T}(\la)$, define $\mu(\fs) \in \mathcal{T}(\la,\mu)$ to be the tableau obtained by replacing each integer $i\geq 1$ with its row index in $\ft^\mu$. 

Suppose that $\la$ is a partition of $n$ and that $\mu$ is a composition of $n$.  
If $\S\in \mathcal{T}_{\text{r}}(\la,\mu)$ define $\Theta_\S:M^\mu \rightarrow S^\la$ to be the homomorphism determined by 
\[\Theta_\S(m_\mu) = \hla+\sum_{{\fs \in \rowstd(\la) \atop \mu(\fs)=\S}}m_\la T_{d(\fs)}. \]
Let $\EHom_{\h}(M^\mu,S^\la)$ be the subspace of $\Hom_{\h}(M^\mu,S^\la)$ consisting of homomorphisms $\Theta$ such that $\Theta=\pi_\la \circ \tilde \Theta$ for some $\tilde\Theta:M^\mu \rightarrow M^\la$.  By construction, if $\S\in \mathcal{T}_{\text{r}}(\la,\mu)$ then $\Theta_\S \in \EHom_{\h}(M^\mu,S^\la)$.  
If $e\geq 2$, say that a partition $\la$ is $e$-restricted if $\la_i - \la_{i+1} <e$ for all $i$.  

\begin{theorem}[{\cite[Corollary 8.7]{DJ:qWeyl}}] \label{SSBasis}
The maps
\[ \{\Theta_S \mid S \in \mathcal{T}_0(\lambda, \mu)\}\]
are a basis of $\EHom_\h(M^\mu,S^\la)$.  Furthermore, unless $q=-1$ and $\lambda$ is not 2-restricted \[\EHom_\h(M^\mu,S^\la)= \Hom_\h(M^\mu, S^\lambda).\]  
\end{theorem}
    
Note that this implies that if $\EHom_\h(M^\mu,S^\la) \neq \{0\}$ then $\la \unrhd \mu$.  \\

Now suppose $\mu$ is a partition and let $\EHom_{\h}(S^\mu,S^\la)$ be the set of maps $\bar\Theta \in \Hom_{\h}(S^\mu,S^\la)$ with the property that $\bar\Theta \circ \pi_\mu \in \EHom_{\h}(M^\mu,S^\la)$.  Again, $\EHom_{\h}(S^\mu,S^\la) = \Hom_{\h}(S^\mu,S^\la)$ unless $q=-1$ and $\la$ is not 2-restricted.  Apart from the fact that our techniques are well-adapted to determining $\EHom_{\h}(S^\mu,S^\la)$, we have another reason to want to study this space.

\begin{theorem}[{\cite[Corollary 8.6]{DJ:qWeyl}}] \label{Weyl}
Let $\mathcal{S}_q$ denote the $q$-Schur algebra and let $\Delta^\mu$ and $\Delta^\la$ be the Weyl modules corresponding to the partitions $\mu$ and $\la$.  Then
\[\EHom_{\h}(S^\mu,S^\la) \cong_R \Hom_{\mathcal{S}_q}(\Delta^\mu,\Delta^\la).\]
\end{theorem}  

Fix a pair of partitions $\la$ and $\mu$ of $n$.  We want to compute $\EHom_{\h}(S^\mu,S^\la)$.  If $\bar\Theta \in \EHom_{\h}(S^\mu,S^\la)$ then $\bar\Theta$ can be pulled back to give a homomorphism $\Theta \in \EHom_{\h}(M^\mu,S^\la)$.  Conversely, $\Theta \in \EHom_\h(M^\mu,S^\la)$ factors through $S^\mu$ if and only if $\Theta(m_\mu h)=0$ for all $h \in \h$ such that $m_\mu h \in \hmu$.  

\begin{diagram}
M^\mu &\rTo^{\pi_\mu} & S^\mu \\
&\rdTo_{\Theta} & \dTo_{\bar\Theta} \\
&& S^\la 
\end{diagram}

We would therefore like to make it easier to check this condition.  
If $\eta=(\eta_1,\eta_2,\ldots,\eta_l)$ is any composition and $k \geq 0$, let $\bar\eta_k=\sum_{i=1}^k \eta_i$ and $\bar{\eta}=\sum_{i \geq 1} \eta_i$.  For $m\geq 0$, let $\sym_{\{m+1,\ldots,m+\bar\eta\}}$ denote the symmetric group on the letters $m+1,\ldots,m+\bar\eta$ and let $\mathcal{D}_{m,\eta}$ be the set of minimal length right coset representatives of $\sym_{(m,\eta_1,\dots,\eta_l)}\cap\sym_{\{m+1,\ldots,m+\bar\eta\}}$ in $\sym_{\{m+1,\ldots,m+\bar\eta\}}$. (Hence if $\ft$ is the $\eta$-tableau with the numbers $m+1,\ldots,m+\bar{\eta}$ entered in order along its rows then $w \in \mathcal{D}_{m,\eta}$ if and only if $\ft w$ is row-standard.)  Set
\[\CC(m;\eta)=\CC(m;\eta_1,\eta_2,\ldots,\eta_l)=\sum_{w \in \mathcal{D}_{m,\eta}} T_w.\]
If $\mu=(\mu_1,\mu_2,\ldots,\mu_b)$ then for $1 \leq d < b$ and $1 \leq t \leq \mu_{d+1}$, define 
\[h_{d,t}= \CC(\bar\mu_{d-1};\mu_d,t).\]

\begin{ex}
Let $\mu = (3,2,2)$.  Then
\begin{align*}
h_{1,1} & = I + T_3 + T_3T_2 + T_3T_2T_1, \\ 
h_{1,2} & = I + T_3 + T_3T_2 + T_3T_2T_1 + T_3T_4 + T_3T_2T_4 + T_3 T_2 T_1 T_4 + T_3T_2T_4T_3+ T_3T_2T_1T_4T_3 + T_3T_2T_1T_4T_3T_2, \\
h_{2,1} &= I + T_5 + T_5T_4, \\
h_{2,2} & = I + T_5 + T_5T_4 +T_5 T_6 +T_5T_4T_6 +T_5T_4T_6T_5. \\
\end{align*}
\end{ex}

\begin{theorem} \label{hdtthm}
Let $\mathcal{I}$ be the right ideal generated by $\{m_\mu h_{d,t} \mid 1 \leq d <b \text{ and } 1 \leq t \leq \mu_{d+1}\}$.  Then  
\[\mathcal{I}=M^\mu \cap \hmu.\]
\end{theorem}

We prove Theorem~\ref{hdtthm} in Section~\ref{hdtproof}.  

\begin{corollary} \label{simple}
Suppose that $\Theta:M^\mu \rightarrow S^\la$.  Then $\Theta(m_\mu h)=0$ for all $h \in \h$ such that $m_\mu h \in \hmu$ if and only if $\Theta(m_\mu h_{d,t})=0$ for all $1 \leq d<b$ and $1 \leq t \leq \mu_{d+1}$.  
\end{corollary}

\begin{remark}
We have chosen to work with the Specht modules which arise as the cell modules for the Murphy basis, rather than the Specht modules of Dipper and James. This is consistent, for example, with the work of Corlett on homomorphisms between Specht modules for the Ariki-Koike algebras~\cite{Corlett}.  As such, the kernel intersection theorem~\cite[Theorem~3.6]{DJ:qWeyl} does not apply, and so Theorem~\ref{hdtthm} has been created to take its place.  In Section~\ref{DJSpecht}, we show that working in either world gives the same results.    
\end{remark}

We have shown that determining $\EHom_{\h}(S^\mu,S^\la)$ is equivalent to finding \[\Psi(\mu,\la)=\{\Theta \in \EHom_{\h}(M^\mu,S^\la) \mid \Theta(m_\mu h_{d,t})=0 \text{ for all } 1 \leq d<b, \;1 \leq t \leq \mu_{d+1}\}, \]
bearing in mind that $\EHom_{\h}(M^\mu,S^\la)$ has a basis indexed by semistandard $\la$-tableaux of type $\mu$.  

If $\S$ is a tableau and $X$ and $Y$ are sets of
positive integers we define~$\S^X_Y$ be the number of entries in row $r$
of~$\S$, for some $r\in Y$, which are equal to some $x \in X$.  We further abbreviate this notation by setting 
$\S^{\le x}_{>r}=\S_{(r,\infty)}^{[1,x]}$, $\S^x_r=\S^{\{x\}}_{\{r\}}$ 
and so on.
Now if $m \geq 0$ define 
\[ [m] = 1+q+\ldots + q^{m-1} \in R.\]
Let $[0]!=1$ and for $m \geq 1$, set $[m]! = [m][m-1]!$.  If $m \geq k \geq 0$, set 
\[\gauss{m}{k} = \frac{[m]!}{[k]![m-k]!}.\]
If $m,k \in \Z$ and any of the conditions $m \geq k \geq 0$ fail, set $\gauss{m}{k}=0$.  We record some results which we will need later. The first is well-known.  

\begin{lemma} \label{GaussSum}
Suppose $n, m \geq 0$.  Then
\begin{align*}
\gauss{n+1}{m} & = \gauss{n}{m-1} + q^{m} \gauss{n}{m} \\
&=\gauss{n}{m} + q^{n-m+1}\gauss{n}{m-1}. 
\end{align*}
\end{lemma}

\begin{lemma} \label{GaussLemma}
Suppose $m, k \geq n \geq 0$.  Then,
\[\sum_{\gamma \geq 0}(-1)^\gamma q^{\binom{\gamma}{2}}\gauss{n}{\gamma}\gauss{m-\gamma}{k} = q^{n(m-k)}\gauss{m-n}{k-n}.\]
\end{lemma}

\begin{proof}
The lemma is true for $n=0$ so suppose $n>0$ and that the lemma holds for $n-1$.
Then  using Lemma~\ref{GaussSum} and the inductive hypothesis,
\begin{align*}
\sum_{\gamma \geq 0}(-1)^\gamma q^{\binom{\gamma}{2}}\gauss{n}{\gamma}\gauss{m-\gamma}{k}
& = \sum_{\gamma\geq 0}(-1)^\gamma q^{\binom{\gamma}{2}} \left(\gauss{n-1}{\gamma} +q^{n-\gamma}\gauss{n-1}{\gamma-1} \right) \gauss{m-\gamma}{k} \\
&=  \sum_{\gamma \geq 0}(-1)^\gamma q^{\binom{\gamma}{2}} \gauss{n-1}{\gamma}\gauss{m-\gamma}{k} - q^{n-1} \sum_{\gamma\geq 0} (-1)^\gamma q^{\binom{\gamma}{2}} \gauss{n-1}{\gamma} \gauss{m-\gamma-1}{k} \\
&= q^{(n-1)(m-k)} \gauss{m-n+1}{k-n+1} -q^{n-1} q^{(n-1)(m-k-1)}\gauss{m-n}{k-n+1} \\
&=q^{n(m-k)} \gauss{m-n}{k-n}
\end{align*}
\end{proof}

The following result may be seen by applying~\cite[Equation 4.6]{M:ULect} and the anti-isomorphism $\ast$ to~\cite[Proposition~2.14 ]{Lyle:CP}.

\begin{proposition}[\cite{Lyle:CP}, Proposition~2.14]  \label{CombTheorem1b}
Suppose that $\T \in \RowT(\la,\mu)$.   
Choose $d$ with $1 \leq d <b$ and $t$ with $1 \leq t \leq \mu_{d+1}$.  Let $\mathcal{S}$ be the set of row-standard tableaux obtained by replacing $t$ of the entries in $\T$ which are equal to $d+1$ with $d$.  Each tableaux $\S \in \mathcal{S}$ will be of type $\nu(d,t)$ where
\[\nu(d,t)_j = 
\begin{cases}
\mu_j+t, & j=d, \\
\mu_j-t, & j=d+1, \\
\mu_j, & \text{otherwise}.
\end{cases} \\ \]  
Recall that $\Theta_\T: M^\mu \rightarrow S^\la$ and $\Theta_{\S}:M^{\nu(d,t)} \rightarrow S^\la$.  Then 
\[\Theta_\T(m_\mu h_{d,t}) = \sum_{\S \in \mathcal{S}} \left(  \prod_{j= 1}^a q^{\T^d_{>j}(\S^d_j - \T^d_j)} \gauss{\S^d_j}{\T^d_j}\right) \Theta_\S (m_{\nu(d,t)}).\] 
\end{proposition}

So if $\Theta \in \EHom_{\h}(M^\mu,S^\la)$ then we may write $\Theta(m_\mu h_{d,t}) = \Phi(m_{\nu(d,t)})$ where $\Phi$ is a linear combination of homomorphisms indexed by $\la$-tableaux of type $\nu(d,t)$, with known coefficients.  However, since these tableaux may not be semistandard, the corresponding homomorphisms may not be linearly independent and so we cannot say immediately whether $\Theta(m_\mu h_{d,t})=0$.  We would therefore like a method of writing a map $\Theta_\S$ as a linear combination of homomorphisms indexed by semistandard tableaux.  Unfortunately, we do not have an algorithm for this process.  However, we do have a way of rewriting homomorphisms.  The following result is due to Fayers and Martin, and holds when $\h \cong R\sym_n$.  It was probably the strongest combinatorial result they used to give their elementary proof of the Carter-Payne theorem~\cite{FayersMartin:homs}.   Recall that if $\eta=(\eta_1,\eta_2,\ldots,\eta_l)$ is any sequence of integers then $\bar\eta_k = \sum_{i=1}^k \eta_i$.   

\begin{proposition}[\cite{FayersMartin:homs}, Lemma 7] \label{FM:Lemma}
Suppose $\h \cong R\sym_n$ and that $\la=(\la_1,\ldots,\la_a)$ is a partition of $n$ and $\nu=(\nu_1,\ldots,\nu_b)$ a composition of $n$.  Suppose $\S \in \RowT(\la,\nu)$.  
Choose $r_1\neq r_2$ with $1 \leq r_1,r_2 \leq a$ and $\la_{r_1} \geq \la_{r_2}$ and $d$ with $1 \leq d \leq b$.  Let
\[\mathcal{G} =\left\{g=(g_1,g_2,\ldots,g_b) \mid g_d=0, \, \bar{g} = \S^d_{r_2},\, \text{ and } g_i \leq \S^{i}_{r_1} \text{ for } 1 \leq i \leq b\right\}.\]
For $g \in \mathcal{G}$, let $\U_g$ be the row-standard tableau formed by moving all entries equal to $d$ from row $r_2$ to row $r_1$ and for $i \neq d$ moving $g_i$ entries equal to $i$ from row $r_1$ to row $r_2$ (where we assume we may reorder the rows if necessary).  Then
\[\Theta_\S = (-1)^{\S^d_{r_2}} \sum_{g \in \mathcal{G}}\left( \prod_{i=1}^b \binom{\S^i_{r_2}+g_i}{g_i}\right) \Theta_{\U_g}.\]
\end{proposition}

Since Fayers and Martin work in the setting by James~\cite{James}, it is not immediate that their result carries over to our cellular algebra setting.  See, however, Section~\ref{DJSpecht}.  

Unfortunately, the obvious $q$-analogue of Proposition~\ref{FM:Lemma}, that is, in the notation above, that the $\h$-homomorphism $\Theta_{\S}$ can be writen as a linear combination of maps $\Theta_{\U_g}$ where $g \in \mathcal{G}$, is false.  The following identity can be checked by hand.  
We identify a tableau $\U$ of type $\nu$ with the image $\Theta_\U(m_\nu)$. 
\begin{ex} Let $\la=(2^2,1)$ and $\nu = (1^5)$.  Then 
\[ \tab(12,35,4) = - \tab(14,35,2)- \tab(24,35,1)+(q-1) \tab(12,34,5).\]
Now 
\[\tab(14,35,2) = - \tab(14,25,3) + \tab(12,34,5) + \tab(13,24,5), \qquad \qquad \tab(24,35,1) =  (q-2) \, \tab(12,34,5) - \tab(12,35,4) - \tab(13,24,5) +  \tab(14,25,3),\]
so that if $q \neq 1$, we cannot write 
$\tab(12,35,4)$ 
as a linear combination of 
$\tab(14,35,2)$ 
and  
$\tab(24,35,1)$.
\end{ex}

We do however have the following weaker analogue of Proposition~\ref{FM:Lemma}. 

\begin{theorem} \label{Lemma7}
Suppose $\la=(\la_1,\ldots,\la_a)$ is a partition of $n$ and $\nu=(\nu_1,\ldots,\nu_b)$ is a composition of $n$.
Let $\S \in \RowT(\la,\nu)$.  
\begin{enumerate}
\item 
Suppose $1 \leq r\leq a-1$ and that $1 \leq d \leq b$.  Let
\[\mathcal{G} =\left\{g=(g_1,g_2,\ldots,g_b) \mid g_d=0, \, \bar{g}=\S^d_{r+1} \text{ and } g_i \leq \S^{i}_{r} \text{ for } 1 \leq i \leq b\right\}.\]
For $g \in \mathcal{G}$, let $\U_g$ be the row-standard tableau formed by moving all entries equal to $d$ from row $r+1$ to row $r$ and for $i \neq d$ moving $g_i$ entries equal to $i$ from row $r$ to row $r+1$. Then
\[\Theta_\S = (-1)^{\S^d_{r+1}} q^{-\binom{\S^d_{r+1}+1}{2}} q^{-\S^d_{r+1}S^{<d}_{r+1}} \sum_{g \in \mathcal{G}} q^{\bar{g}_{d-1}} \prod_{i=1}^b q^{g_i \S^{<i}_{r+1}} \gauss{\S^i_{r+1}+g_i}{g_i}\Theta_{\U_g}.\]
\item 
Suppose $1 \leq r\leq a-1$ and $\la_r=\la_{r+1}$ and that $1 \leq d \leq b$.  Let
\[\mathcal{G} =\left\{g=(g_1,g_2,\ldots,g_b) \mid g_d=0, \, \bar{g} = \S^d_r \text{ and } g_i \leq \S^{i}_{r+1} \text{ for } 1 \leq i \leq b \right\}.\]
For $g \in \mathcal{G}$, let $\U_g$ be the row-standard tableau formed by moving all entries equal to $d$ from row $r$ to row $r+1$ of $\S$ and for $i \neq d$ moving $g_i$ entries equal to $i$ from row $r+1$ to row $r$. Then
\[\Theta_\S =  (-1)^{\S^d_{r}} q^{-\binom{\S^d_{r}}{2}} q^{-\S^d_r \S^{>d}_r} \sum_{g \in \mathcal{G}} q^{-\bar{g}_{d-1}}  \prod_{i=1}^b q^{g_i \S^{>i}_{r}} \gauss{\S^i_{r}+g_i}{g_i} \Theta_{\U_g}.\]
\end{enumerate}
\end{theorem}

The proof of Theorem~\ref{Lemma7} is both technical and long, so we postpone it until the next section and give some examples.  As above, we identify a tableau $\T$ of type $\sigma$ with $\Theta_{\T}(m_\sigma)$.  Set $$e=\min\{f\ge2 \mid 1+q+\dots+q^{f-1}=0\},$$
with $e=\infty$ if $1+q+\dots+q^{f-1}\ne0$ for all $f\ge2$, and recall that if $R$ is a field then $\h$ is (split) semisimple if and only if $e>n$.  

\begin{ex}
Suppose $R$ is a field.  Let $\mu=(3,2,2)$ and $\la=(5,2)$.  If $\Theta \in \EHom_\h(M^\mu, S^\la)$ then $\Theta$ is determined by
\[\Theta(m_\mu)= \alpha \tab(11122,33) + \beta \tab(11123,23)+ \gamma \tab(11133,22)\]
for some $\alpha,\beta,\gamma \in R$.  
Then applying Proposition~\ref{CombTheorem1b} and Theorem~\ref{Lemma7}, 
\begin{align*}
\Theta(m_\mu h_{2,1}) & = \alpha \tab(11122,23) + [2]\beta \tab(11123,22)+  q[2]\beta \tab(11122,23) + q^2 \gamma \tab(11123,22), \\
\Theta(m_\mu h_{2,2}) & = \alpha \tab(11122,22)+ q[2]^2 \beta \tab(11122,22) + q^4 \gamma \tab(11122,22), \\
\Theta(m_\mu h_{1,1}) & = [4] \alpha \tab(11112,33) + [4] \beta \tab(11113,23) + \beta \tab(11123,13) + \gamma \tab(11133,12) \\
&= [4] \alpha \tab(11112,33) + [4] \beta \tab(11113,23) - \beta \tab(11113,23) - [2] \beta \tab(11112,33) - q \gamma \tab(11113,23), \\
\Theta(m_\mu h_{1,2}) & = \gauss{5}{2} \alpha \tab(11111,33) + [4] \beta \tab(11113,13) + \gamma \tab(11133,11) \\
&= \gauss{5}{2} \alpha \tab(11111,33) -[2][4]\beta \tab(11111,33) + q \gamma \tab(11111,33). 
\end{align*}
Since $\Theta \in \Psi(\mu,\la)$ if and only if $\Theta(m_\mu h_{d,t})=0$ for all $d,t$ above, we see that $\Psi(\mu,\la)$ and hence
$\EHom_\h(S^\mu,S^\la)$ is 1-dimensional if $e=5$ and zero otherwise.  Moreover if $e=5$, the space $\Psi(\mu,\la)$ is spanned by the map $\Theta$ determined by
\[\Theta(m_\mu)= q^3 [2] \tab(11122,33) -q^2 \tab(11123,23)+ [2] \tab(11133,22).\]
\end{ex}

\begin{ex}
Suppose that $q=1$ and that $R$ is a field of characteristic 2.  Let $\la=(10,5)$ and $\mu=(8,3,1,1,1,1)$.  Then $\dim(\EHom_\h(S^\mu,S^\la))=2$.  If $\T=\tab(1111111122,23456)$ then the space $\Psi(\mu,\la)$ is spanned by the maps
$\Theta_{\T}$ and 
$\sum_{\S \in \mathcal{T}_0(\la,\mu)} \Theta_\S$.
\end{ex}

It has recently been shown that for fixed parameters $e$ and $p$ there exist homomorphism spaces of arbitrarily high dimension~\cite{Dodge:Dim,Lyle:Dim}.   

\begin{ex}
Take $\la=(5,4)$ and $\mu = (3,3,2,1)$.  Then $\T=\tab(11133,2224) \in \mathcal{T}_0(\la,\mu)$ and 
\[ \Theta_\T(m_\mu h_{3,1}) = \tab(11133,2223).\]
But we now have no obvious way of using Theorem~\ref{Lemma7} to write $\tab(11133,2223)$ in terms of homomorphisms indexed by semistandard tableaux.  
 \end{ex} 

We are therefore most interested in pairs of partitions $\la$ and $\mu$ where Proposition~\ref{CombTheorem1b} and Theorem~\ref{Lemma7} can give an algorithm for computing $\EHom_{\h}(S^\mu,S^\la)$.  Suppose $\T \in \mathcal{T}_0(\la,\mu)$.  If $1 \leq d <b$ and $1 \leq t \leq \mu_{d+1}$ then Theorem~\ref{SSBasis} and Proposition~\ref{CombTheorem1b} show that there exist unique $m_{\U\T} \in R$ such that
\begin{equation} \label{MUT}
\Theta_\T(m_\mu h_{d,t}) = \sum_{\U \in \mathcal{T}_0(\la,\nu(d,t))} m_{\U\T} \Theta_\U(m_{\nu(d,t)}),
\end{equation}
where $\nu(d,t)$ is defined in Proposition~\ref{CombTheorem1b}.  Now suppose $R$ is a field.  
If $M=(m_{\U\T})$ is the matrix whose columns are indexed by tableaux $\T \in\mathcal{T}_0(\la,\mu)$ and rows by tableaux $\U \in \mathcal{T}_0(\la,\nu(d,t))$ for some $1 \leq d <b$ and $1 \leq t \leq \mu_{d+1}$, with entries $m_{\U\T}$ as in Equation~\ref{MUT} then, by Corollary~\ref{simple},
\[\dim(\EHom_{\h}(S^\mu,S^\la)) = \dim(\Psi(\mu,\la))=\corank(M).\]     
So the outstanding problem is to determine an explicit formula for $m_{\U\T}$.  For the remainder of Section~\ref{HomSection}, suppose that $R$ is a field and that  
$\la=(\la_1,\ldots,\la_a)$ and $\mu=(\mu_1,\ldots,\mu_b)$ satisfy
\[\bar\mu_j \geq \bar\la_{j-1}+\la_{j+1} \text{ for } 1 \leq j < a.\]  
Since $\EHom_\h(S^\la,S^\mu)=\{0\}$ if $a>b$, we may also assume that $a \leq b$.  
Let $a \leq d <b$ and $1 \leq t \leq \mu_{d+1}$.  If $\T \in \mathcal{T}_0(\la,\mu)$ then say that $\T \xrightarrow{d,t}\U$ if 
\begin{itemize}
\renewcommand{\labelitemi}{$-$}
\item $\U \in \RowT(\la,\nu(d,t))$;  
\item $\U^i_j = \T^i_j$ for all $i \neq d,d+1$ and all $j$; 
\item $\U^d_j \geq \T^d_j$ for all $j$.  
\end{itemize}

\begin{lemma} \label{abig}
Let $\T \in \mathcal{T}_0(\la,\mu)$.  Suppose that $a\leq d <b$ and $1 \leq t \leq \mu_{d+1}$.  Then
\[\Theta_\T(m_\mu h_{d,t})= \sum_{\T\xrightarrow{d,t}\U} \left(  \prod_{j= 1}^a q^{\T^d_{>j}(\U^d_j - \T^d_j)} \gauss{\U^d_j}{\T^d_j}\right) \Theta_\U(m_{\nu(d,t)}) \]
and if $\T\xrightarrow{d,t}\U$ then $\U$ is semistandard.
\end{lemma}

\begin{proof}
Since $\T \xrightarrow{d,t}\U$ precisely when $\U$ is a row-standard tableau formed from $\T$ by changing $t$ entries equal to $d+1$ into $d$, the first part of the lemma is a restatement of Proposition~\ref{CombTheorem1b}.  So suppose $\T \xrightarrow{d,t}\U$.  If $1 \leq r < a$, then $\T^r_j=0$ for $j>r$, since $\T$ is semistandard, and so 
\begin{align*}
\T^r_r & = \mu_r -\T^r_{<r} \\
& =\mu_r -(\bar\la_{r-1}-\T_{<r}^{<r} -\T_{<r}^{>r} ) \\
&\geq \mu_r - \bar\la_{r-1}+\T^{<r}_{<r}\\
&=\bar\mu_r-\bar\la_{r-1} \\
&\geq \la_{r+1}
\end{align*}
where the last inequality comes from our assumption on $\la$ and $\mu$.  
Each row $1 \leq r <a$ therefore contains at least $\la_{r+1}$ entries equal to $r$ and so each entry equal to $d+1$ in $\T$ is either in the top row or lies in row $r+1$ for some $1 \leq r<a$ and so is directly below a node of residue $r<d$.  Hence a row-standard tableau obtained by changing entries equal to $d+1$ into $d$ is semistandard.  
\end{proof}

Let $1 \leq d <a$ and $1 \leq t \leq \mu_{d+1}$.  If $\T \in \mathcal{T}_0(\la,\mu)$ then say that $\T \xrightarrow{d,t}\U$ if $\U \in \RowT(\la,\nu(d,t))$ and
\[\begin{array}{c|ccc}
& i=d & i=d+1 & i \neq d,d+1 \\ \hline
j=d &\U^i_j \geq \T^i_j& \U^i_j \leq \T^i_j & \U^i_j \leq \T^i_j\\
j=d+1 &\U^i_j =0 &\U^i_j \leq \T^i_j& \U^i_j \geq \T^i_j \\
j \neq d,d+1 &\U^i_j \geq \T^i_j&\U^i_j \leq \T^i_j& \U^i_j = \T^i_j
\end{array}\]
Of course, some of the conditions in the table above are redundant since they are implied by the others. 

\begin{lemma}
Let $\T \in \mathcal{T}_0(\la,\mu)$.  Suppose that $1\leq d <a$ and $1 \leq t \leq \mu_{d+1}$.  Then

\begin{multline*}
\Theta_\T(m_\mu h_{d,t})= \sum_{\T\xrightarrow{d,t}\U} 
(-1)^{\T^{d+1}_{d+1}-\U^{d+1}_{d+1}} 
q^{-\binom{\T^{d+1}_{d+1}-\U^{d+1}_{d+1}+1}{2}}q^{\U^{d+1}_{d+1}(\U^d_d-\T^{d}_{d}+\U^{d+1}_{d+1}-\T^{d+1}_{d+1})} \\
\left( \prod_{j=1}^{d-1} q^{\T^d_{>j}(\U^d_j-\T^d_j)} \gauss{\U^d_j}{\T^d_j} \right)\gauss{\U^d_d-\T^{d+1}_{d+1}}{\T^d_d-\U^{d+1}_{d+1}}  \left( \prod_{i=d+2}^{b} q^{\T^{<i}_{d+1}(\U^i_{d+1}-\T^i_{d+1})} \gauss{\U^i_{d+1}}{\T^i_{d+1}} \right) \Theta_\U(m_{\nu(d,t)}) 
\end{multline*}
and if $\T\xrightarrow{d,t}\U$ then $\U$ is semistandard.
\end{lemma}

\begin{proof}
By Proposition~\ref{CombTheorem1b}, $\Theta_\T(m_\mu) = \sum_\S a_\S \Theta_\S(m_{\nu(d,t)})$ where the sum is over row-standard tableaux formed by changing $t$ entries equal to $d+1$ in $\T$ into $d$.  Suppose $\S$ is such a tableau.  By Theorem~\ref{Lemma7}, if $\S^d_{d+1}>\la_d-\S^d_d$ then $\Theta_\S=0$; otherwise $\Theta_\S$ is a linear combination of maps $\Theta_\U$ indexed by row-standard tableaux $\U$ where $\U$ is formed from $\S$ by moving $\S^d_{d+1}$ entries equal to $d$ from row $d+1$ to row $d$ and replacing them with entries not equal to $d$ from row $d$.  If $\U$ has this form, then clearly $\T\xrightarrow{d,t}\U$ so that 
\[\Theta_\T(m_\mu h_{d,t})= \sum_{\T\xrightarrow{d,t}\U} b_\U \Theta_\U(m_{\nu(d,t)}) \]
for some $b_\U \in R$.  So suppose $\T\xrightarrow{d,t}\U$.  Each of the intermediate tableaux $\S$ were formed from $\T$ by changing entries of $\T$ from $d+1$ to $d$.  Then $\U^d_j-\T^d_j$ entries were changed in row $j$ for $1\leq j \leq d-1$, and for some $\gamma \geq 0$, $\gamma$ entries were changed in row $d$ and $\U^d_d-\T^d_d-\gamma$ entries in row $d+1$.  Therefore, summing over all $\gamma$ with $\max\{0,\U^d_d-\T^d_d-\T^{d+1}_{d+1}\} \leq \gamma \leq \U^d_d-\T^d_d+\U^{d+1}_{d+1}-\T^{d+1}_{d+1}$, we have
\begin{align*}
b_\U & = \sum_{\gamma} 
\left( \prod_{j=1}^{d-1} q^{\T^d_{>j}(\U^d_j-\T^d_j)} \gauss{\U^d_j}{\T^d_j} \right) 
\gauss{\T^d_d + \gamma}{\gamma}
(-1)^{\U^d_d-\T^d_d-\gamma} q^{-\binom{\U^d_d-\T^d_d-\gamma}{2}-\U^d_d+\T^d_d+\gamma} \\
& \hspace*{15mm} q^{(\U^{d}_{d}-\T^{d}_{d}+\U_{d+1}^{d+1}-\T^{d+1}_{d+1}-\gamma)(\U^d_d-\T^d_d-\gamma)}\gauss{\U^{d+1}_{d+1}}{\T^d_d-\U_d^d+\T^{d+1}_{d+1}+\gamma} 
\left( \prod_{i=d+2}^{b} q^{\T^{<i}_{d+1}(\U^i_{d+1}-\T^i_{d+1})} \gauss{\U^i_{d+1}}{\T^i_{d+1}} \right) \\
&= (-1)^{\U^d_d-\T^d_d}\left( \prod_{j=1}^{d-1} q^{\T^d_{>j}(\U^d_j-\T^d_j)} \gauss{\U^d_j}{\T^d_j} \right) \left( \prod_{i=d+2}^{b} q^{\T^{<i}_{d+1}(\U^i_{d+1}-\T^i_{d+1})} \gauss{\U^i_{d+1}}{\T^i_{d+1}} \right) \\
&  \hspace*{15mm} q^{(\U^{d+1}_{d+1}-\T^{d+1}_{d+1})(\U^d_d-\T^d_d) + \binom{\U_d^d-\T^d_d}{2}} \sum_{\gamma} (-1)^\gamma q^{\binom{\gamma+1}{2}-\gamma(\U^d_d-\T^d_d+\U^{d+1}_{d+1}-\T^{d+1}_{d+1})}\gauss{\T^d_d + \gamma}{\gamma} \gauss{\U^{d+1}_{d+1}}{\T^d_d-\U_d^d+\T^{d+1}_{d+1}+\gamma}
\end{align*}

Now we change the limits on the sum and apply Lemma~\ref{GaussLemma}.
\begin{align*}
\sum_{\gamma}(-1)^\gamma &  
 q^{\binom{\gamma+1}{2}-\gamma(\U^d_d-\T^d_d+\U^{d+1}_{d+1}-\T^{d+1}_{d+1})}
\gauss{\T^d_d + \gamma}{\gamma} \gauss{\U^{d+1}_{d+1}}{\T^d_d-\U_d^d+\T^{d+1}_{d+1}+\gamma} \\
&=(-1)^{\U_d^d-\T_d^d+\U^{d+1}_{d+1}-\T^{d+1}_{d+1}} q^{-\binom{\U_d^d-\T^d_d+\U^{d+1}_{d+1}-\T_{d+1}^{d+1}}{2}} \sum_{\gamma}(-1)^\gamma q^{\binom{\gamma}{2}} \gauss{\U^{d+1}_{d+1}}{\gamma} \gauss{\U^{d}_{d}+\U^{d+1}_{d+1} -\T^{d+1}_{d+1}-\gamma}{\T^d_d} \\
&= (-1)^{\U_d^d-\T_d^d+\U^{d+1}_{d+1}-\T^{d+1}_{d+1}} q^{-\binom{\U_d^d-\T^d_d+\U^{d+1}_{d+1}-\T_{d+1}^{d+1}}{2}} q^{\U^{d+1}_{d+1}(\U^d_d-\T^d_d+\U^{d+1}_{d+1}-\T^{d+1}_{d+1})} \gauss{\U^d_d-\T^{d+1}_{d+1}}{\T^d_d-\U^{d+1}_{d+1}} 
\end{align*}

Hence
\begin{multline*}
b_\U = (-1)^{\T^{d+1}_{d+1}-\U^{d+1}_{d+1}} 
q^{-\binom{\T^{d+1}_{d+1}-\U^{d+1}_{d+1}+1}{2}}q^{\U^{d+1}_{d+1}(\U^d_d-\T^{d}_{d}+\U^{d+1}_{d+1}-\T^{d+1}_{d+1})} \\
\left( \prod_{j=1}^{d-1} q^{\T^d_{>j}(\U^d_j-\T^d_j)} \gauss{\U^d_j}{\T^d_j} \right)\gauss{\U^d_d-\T^{d+1}_{d+1}}{\T^d_d-\U^{d+1}_{d+1}}  \left( \prod_{i=d+2}^{b} q^{\T^{<i}_{d+1}(\U^i_{d+1}-\T^i_{d+1})} \gauss{\U^i_{d+1}}{\T^i_{d+1}} \right) .
\end{multline*}

It remains to show that if $\T\xrightarrow{d,t}\U$ then $\U$ is semistandard.  From the proof of Lemma~\ref{abig}, we have that $\T^r_r \geq \la_{r+1}$ for $1 \leq r <a$, so the only way that $\U$ can fail to be semistandard is if there is an entry in row $d+1$ which is as big or bigger than the entry directly below it.  However,  
\begin{align*}\U^{d+1}_{d+1}
& =\mu_{d+1}-t -\U^{d+1}_{\leq d} \\
&=\mu_{d+1} -t-(\bar\la_{d}-\U_{\leq d}^{\leq d} - \U_{\leq d}^{>d+1})\\
&\geq \mu_{d+1} - t-\bar\la_{d}+\U^{\leq d}_{\leq d} \\
&=\bar\mu_{d+1}-\bar\la_{d} \\
&\geq \la_{d+2}
\end{align*}
so this is not possible.  
\end{proof}

Let us summarize the results above.  

\begin{proposition} \label{algorithm}
Suppose $\la=(\la_1,\ldots,\la_a)$ and $\mu=(\mu_1,\ldots,\mu_b)$ are partitions of $n$ with the property that $\bar\mu_j \geq \bar\la_{j-1}+\la_{j+1}$ for $1 \leq j < a$.  Let $M=(m_{\U\T})$ be the matrix whose columns are indexed by tableaux $\T \in\mathcal{T}_0(\la,\mu)$ and rows by tableaux $\U \in \mathcal{T}_0(\la,\nu(d,t))$ for some $1 \leq d <b$ and $1 \leq t \leq \mu_{d+1}$, where
\[m_{\U\T}=
\begin{cases}
(-1)^{\T^{d+1}_{d+1}-\U^{d+1}_{d+1}} 
q^{-\binom{\T^{d+1}_{d+1}-\U^{d+1}_{d+1}+1}{2}}q^{\U^{d+1}_{d+1}(\U^d_d-\T^{d}_{d}+\U^{d+1}_{d+1}-\T^{d+1}_{d+1})}  \gauss{\U^d_d-\T^{d+1}_{d+1}}{\T^d_d-\U^{d+1}_{d+1}} \\
\qquad \times\left( \prod_{j=1}^{d-1} q^{\T^d_{>j}(\U^d_j-\T^d_j)} \gauss{\U^d_j}{\T^d_j} \right)  \left( \prod_{i=d+2}^{b} q^{\T^{<i}_{d+1}(\U^i_{d+1}-\T^i_{d+1})} \gauss{\U^i_{d+1}}{\T^i_{d+1}} \right) , & 1 \leq d <a \text{ and } \T \xrightarrow{d,t} \U, \\
\prod_{j= 1}^a q^{\T^d_{>j}(\U^d_j - \T^d_j)} \gauss{\U^d_j}{\T^d_j}, & a \leq d <b \text{ and } \T \xrightarrow{d,t} \U,\\
0, & \text{otherwise}.
\end{cases}
\]
Then $\dim(\EHom_\h(S^\mu,S^\la))=\corank(M)$.  
\end{proposition}

We have taken a hard problem in representation theory and reduced it to a combination of combinatorics and linear algebra.  However it should however be noted that in doing so we have lost some algebraic information.  For example, $\EHom_\h(S^\mu,S^\la)=\{0\}$ unless $S^\mu$ and $S^\la$ lie in the same block.  Proposition~\ref{algorithm} does not seem to make use of this fact.  

We note that Proposition~\ref{CombTheorem1b} and Theorem~\ref{Lemma7} can be used to compute homomorphism spaces other than those we have considered above.  For example, the proof of the one-node Carter-Payne Theorem in~\cite{Lyle:CP} relied on Proposition~\ref{CombTheorem1b} and some special cases of Theorem~\ref{Lemma7}, but a one-node Carter-Payne pair $\mu$ and $\la$ do not necessarily satisfy $\bar\mu_{j} \geq \bar\la_{j-1}+\la_{j+1}$ for $1 \leq j <a$.  

While a computer can use Proposition~\ref{algorithm} to solve individual problems, it is more satisfying to have explicit results.  This is the purpose of the next section.

\subsection{Explicit homomorphism spaces} \label{Spaces}
In this section we show that a lower bound on $\dim(\Hom_{\h}(S^\mu,S^\la))$ can be obtained by looking at the algebra $\h_{\C,q}(\sym_n)$ and we then give the dimension of $\Hom_{\h_{\C,q}(\sym_n)}(S^\mu,S^\la)$ where $\la=(\la_1,\la_2)$ and $\mu_1 \geq \la_2$.  
We would like to thank Meinolf Geck and Lacrimioara Iancu for pointing out the proof of Proposition~\ref{BiggerDim}.  

Fix a field $k$ of characteristic $p>0$ and let $\eta \in k^\times$.  Let $e=\min\{f\geq 2 \mid 1+\eta+\ldots+\eta^{f-1}=0\}$ where we assume that $e<\infty$.  Let $\omega$ be a primitive $e^{\text{th}}$ root of unity in $\C$.

Let $\mathcal{Z}=\Z[\q,\q^{-1}]$ denote the ring of Laurent polynomials in the indeterminate $\q$.  If $F$ is a field and $q$ an invertible element of $F$, define $\theta_{F,q}: \mathcal{Z} \rightarrow F$ to be the ring homomorphism which sends $\q$ to $q$.  If $S$ is a ring and $M \in M_{l \times m}(S)$, then $\rank(M)$ is the greatest order of any non-zero minor of $M$.  If $\tilde{\theta}:S \rightarrow S'$ is a ring homomorphism, define $\tilde\theta(M)\in M_{l\times m}(S')$ to be the matrix with $(i,j)$-entry $\tilde\theta(M_{ij})$.  For $M \in M_{l\times m}(\mathcal{Z})$ set $M_{F,q} = \theta_{F,q}(M)$.

Let $\Phi_e(X) \subseteq \Z[X]$ denote the $e^{\text{th}}$ cyclotomic polynomial.  Observe that $\theta_{k,\eta}(\Phi_e(\q))=0_k$ and $\theta_{\C,\omega}(\Phi_e(\q))=0_\C$ so that the maps $\theta_{k,\eta}$ and $\theta_{\C,\omega}$ both factor through $R=\Z[\q,\q^{-1}] / (\Phi_e(\q))$, that is, there exist ring homomorphisms $\tilde{\theta}_{k,\eta}$ and $\tilde{\theta}_{\C,\omega}$ such that the following diagram commutes.   

\begin{diagram}
&& k \\
&\ruTo^{\theta_{k,\eta}} & \uTo_{\tilde\theta_{k,\eta}} \\
\mathcal{Z}=\Z[\q,\q^{-1}] &\rTo & \Z[\q,\q^{-1}]/(\Phi_e(\q))=R \\
&\rdTo_{\theta_{\C,\omega}} & \dTo_{\tilde\theta_{\C,\omega}} \\
&& \C 
\end{diagram}

\begin{lemma} \label{LemF}
Suppose $\tilde{M} \in M_{l \times m}(R)$.  Then $\rank(\tilde{M})\geq \rank(\tilde{\theta}_{k,\eta}(\tilde{M}))$.
\end{lemma}

\begin{proof}
This follows since if $N$ is any $r \times r$ submatrix of $M$ then $\det(\tilde{\theta}_{k,\eta}(N)) =\tilde{\theta}_{k,\eta}(\det(N))$.  
\end{proof}

\begin{lemma}
Suppose $\tilde{M} \in M_{l \times m}(R)$.  Then $\rank(\tilde{M})=\rank(\tilde{\theta}_{\C,\omega}(\tilde{M}))$.
\end{lemma}

\begin{proof}
This follows from the proof of Lemma~\ref{LemF} and the fact that $\tilde\theta_{\C,\omega}$ is injective.  
\end{proof}

\begin{corollary} \label{Rank}
 Suppose $M \in  M_{l\times m}(\mathcal{Z})$.  Then
\[\rank(M_{k,\eta}) \leq \rank(M_{\C,\omega}).\]
\end{corollary}

Now let $\h^{\mathcal{Z}}=\h_{\mathcal{Z},\q}(\sym_n)$. 
If $F$ is a field and $q$ an invertible element of $F$ then $\h_{F,q}(\sym_n) \cong \h^{\mathcal{Z}} \otimes_{\mathcal{Z}} F$, where $\q$ acts on $F$ as multiplication by $q$.  If $A$ is a $\h^\mathcal{Z}$-module then we define the $\h_{F,q}(\sym_n)$-module $A_{F,q}=A \otimes_{\mathcal{Z}}F$. 

\begin{proposition} \label{BiggerDim}
Suppose $A$ and $B$ are $\h^{\mathcal{Z}}$-modules which free as $\mathcal{Z}$-modules of finite rank. Then
\[\dim(\Hom_{\h_{k,\eta}(\sym_n)}(A_{k,\eta},B_{k,\eta})) \geq \dim(\Hom_{\h_{\C,\omega}(\sym_n)}(A_{\C,\omega},B_{\C,\omega})).\] 
\end{proposition}

\begin{proof}
Choose bases $\{a_i \mid i \in I\}$ of $A$ and $\{b_j \mid j \in J\}$ of $B$ and let $\{\phi_{ij} \mid  i \in I, j \in J\}$ be the corresponding basis of $\Hom_{\mathcal{Z}}(A,B)$.  Then $\phi =\sum_{i,j} \alpha_{ij} \phi_{ij}$ lies in  $\Hom_{\h^{\mathcal{Z}}}(A,B)$ if and only if the coefficients $\alpha_{ij}$ satisfy a system of equations of the form $\sum_{i,j} \beta^{k}_{ij} \alpha_{ij}=0$ for $1 \leq k \leq N$, some $N\geq 0$.  If we let $M$ be the matrix whose columns are indexed by $\{(i,j) \mid i \in I, j \in J\}$ and rows by $1 \leq k \leq N$ and which has entries $\beta^k_{ij}$ then $\dim(\Hom_{\h^{\mathcal{Z}}}(A,B))= \corank(M)$.  Furthermore, 
\[ \dim(\Hom_{\h_{k,\eta}(\sym_n)}(A_{k,\eta},B_{k,\eta}))= \corank(M_{k,\eta}) \geq \corank(M_{\C,\omega}) =   \dim(\Hom_{\h_{\C,\omega}(\sym_n)}(A_{\C,\omega},B_{\C,\omega}))\]
by Corollary~\ref{Rank}.
\end{proof}

In particular, we may take $A$ and $B$ to be Specht modules.  
 
\begin{corollary} \label{DimCorollary}
Suppose that $\la$ and $\mu$ are partitions of $n$.  Then
\[\dim(\Hom_{\h_{k,\eta}(\sym_n)}(S^\mu_{k,\eta},S^\la_{k,\eta})) \geq \dim(\Hom_{\h_{\C,\omega}(\sym_n)}(S^\mu_{\C,\omega},S^\la_{\C,\omega})).\] 
\end{corollary}

The following result is not implied by Corollary~\ref{DimCorollary} if $e \neq 2$.  It could also be proved by giving an analogue of Proposition~\ref{BiggerDim} for the $q$-Schur algebra and using Theorem~\ref{Weyl}.

\begin{corollary} \label{Charp}
Suppose that $\la$ and $\mu$ are partitions of $n$.  Then
\[\dim(\EHom_{\h_{k,\eta}(\sym_n)}(S^\mu_{k,\eta},S^\la_{k,\eta})) \geq \dim(\EHom_{\h_{\C,\omega}(\sym_n)}(S^\mu_{\C,\omega},S^\la_{\C,\omega})).\] 
\end{corollary}

\begin{proof}
Let $M=(m_{\U\T})$ be the matrix with entries in $\mathcal{Z}$ given by Equation~\ref{MUT}.  Then
\[\dim(\EHom_{\h_{k,\eta}(\sym_n)}(S^\mu_{k,\eta},S^\la_{k,\eta})) =\corank(M_{k,\eta}) \geq \corank(M_{\C,\omega}) = \dim(\EHom_{\h_{\C,\omega}(\sym_n)}(S^\mu_{\C,\omega},S^\la_{\C,\omega})).\] 
\end{proof}

Now fix $2 \leq e < \infty$ and let $\h=\h_{\C,\omega}(\sym_n)$ where $\omega$ is a primitive $e^{\text{th}}$ root of unity in $\C$.  If $\nu$ is a partition of $n$, let $\ell(\nu)$ be the number of non-zero parts of $\nu$.  We describe $\dim(\EHom_\h(S^\mu,S^\la))$ where $\ell(\la) \leq 2$ and $\mu_1 \geq \la_2$.   
Where $\ell(\mu) \leq 3$, please note that the homomorphism space $\EHom_{\h_{F,q}}(S^\mu,S^\la)$ has been computed for arbitary Hecke algebras of type $A$, even for partitions where $\mu_1<\la_2$~\cite{Cox,Parker}. 

\begin{proposition}[~\cite{Cox}] \label{first}
Suppose that $\la=(n)$ and that $\mu=(\mu_1,\ldots,\mu_b)$ is a partition of $n$.  Then 
\[\dim(\EHom_\h(S^\mu,S^\la)) = \begin{cases}
1, & \mu=(n), \\ 
1, & \ell(\mu) \geq 2, \,\mu_1 \equiv -1 \mod e \text{ and } \mu_i=e-1 \text{ for }2 \leq i<b,\\
0, & \text{otherwise}. 
\end{cases}\]
\end{proposition}

For the remainder of Section~\ref{Spaces}, suppose $\la$ and $\mu$ are such that $\ell(\la)=2$, $\ell(\mu)=b$ and $\mu_1 \geq \la_2$.  Since $\EHom_{\h}(S^\mu,S^\la)=\{0\}$ if $\mu_1 > \la_1$ we also assume $\mu_1 \leq \la_1$ and $b\geq 2$.  For $k \geq 1$, let $N^\mu_k=\#\{1 \leq i \leq b \mid \mu_i =k\}$.  If $m\geq 0$, let $m'$ be the integer such that $0\leq m'<e$ and $m\equiv m' \mod e$.    

\begin{proposition} \label{OneDim}We have that
\[\dim(\Hom_{\h}(S^\mu,S^\la))\leq 1.\]
\end{proposition}

\begin{proposition}[~\cite{LM:rowhoms,Donkin:tilting}]
Suppose that $\mu_1 = \la_1$.  Let $\bar\la=(\la_2)$ and $\bar\mu=(\mu_2,\ldots,\mu_b)$.  Then   
\[\dim(\EHom_{\h}(S^\mu,S^\la)) = \dim(\EHom_{\h}(S^{\bar\mu},S^{\bar\la})).\]
\end{proposition}

In fact, this is obvious in our setup since the matrices $M=(m_{\U\T})$ obtained in both cases are the same.  

\begin{proposition}[~\cite{Cox}]
Suppose that $\ell(\mu)=2$ and $\la_1>\mu_1$.  Then $\dim(\EHom_{\h}(S^\mu,S^\la))=1$ if and only if
\begin{itemize}
\item $\la_1 - \mu_1 <e$ and $\mu_1-\la_2+1 \equiv 0$.
\end{itemize}
\end{proposition}

\begin{proposition} [\cite{Parker}]
Suppose that $\ell(\mu)=3$ and $\la_1>\mu_1$.  Then $\dim(\EHom_{\h}(S^\mu,S^\la))=1$ if and only if
\begin{itemize}
\item $\mu_2=e-1$ and $\mu_1-\la_2+1 \equiv 0$ and $\la_2 \leq \mu_3$; or
\item $\mu_2+1 \equiv 0$ and  $\mu_1-\la_2+1 \equiv 0$ and $\la_2 \geq \mu_3$ and $\mu_3 \leq e-1$ and $\la_1-\mu_1<e$; or
\item $\mu_1+2 \equiv 0$ and $\mu_2 = \la_2$ and $\mu_3 \leq e-1$; or
\item $\mu_1+2 \equiv 0$ and $\mu_2 = \la_2$ and $\mu_3 \leq 2e-2$ and $(\la_2+1)'>\mu_3'$; or
\item $\mu_1+2 \equiv 0$ and $\mu_2 \neq \la_2$ and $\la_2>\mu_3$ and $\la_1-\mu_2+1 \equiv 0$ and $\mu_3 \leq e-1$ and $\la_1 - \mu_1 <e$.  
\end{itemize}
\end{proposition}

\begin{proposition}
Suppose that $\ell(\mu) \geq 4$, that $\la_1>\mu_1$ and that $\mu_3=e-1$.   Then $\dim(\EHom_{\h}(S^\la,S^\mu))=1$ if and only if $\mu_{b-1}=e-1$ and 
\begin{itemize}
\item $\mu_2=e-1$ and $\mu_1-\la_2+1 \equiv 0$; or
\item $\mu_2+1 \equiv 0$ and $\mu_1-\la_2+1 \equiv 0$ and $\la_2 \geq \mu_2$ and $\la_1< \mu_1+\mu_2$ and $\la_1-\mu_1 <e$; or
\item $\mu_1+2 \equiv 0$ and $\la_1-\mu_2+1 \equiv 0$ and $\la_2 \geq \mu_2$ and $\la_1< \mu_1+\mu_2$ and $\la_1-\mu_1 <e$; or
\item $\mu_1+2 \equiv 0$ and $\mu_2 \equiv \la_2$ and $\la_2 \geq \mu_2$ and $(\mu_2+1)'\leq \la_1-\mu_1$.  
\end{itemize}  
\end{proposition}

Now say that the partition $\mu$ has a good shape if it satisfies the following properties.

\begin{itemize}
\renewcommand{\labelitemi}{$-$}
\item $\ell(\mu) \geq 4$; and
\item $\mu_1 +2 \equiv \mu_2+2 \equiv 0 \mod e$; and 
\item $\mu_3 \leq 2e-2$; and 
\item $\#\{3 \leq i \leq b \mid \mu_i \neq 2e-2,e-1\} \leq 2$; and 
\item If $N^\mu_{e-1} >0$ then $\#\{3 \leq i \leq b \mid e-1>\mu_i\}\leq 1$ and 
$\#\{3 \leq i \leq b \mid 2e-2>\mu_i>e-1\} \leq 1$. 
\end{itemize}
If $\mu$ has a good shape define $\mu^\ast$ to be the partition given by
\[\mu^\ast = (\mu_1 + (N^\mu_{e-1}+1)(e-1),\mu_2+(N^\mu_{e-1}-1)(e-1)).\]
Let \[\alpha = \min\{3 \leq i \leq b \mid \mu_i \neq 2e-2,e-1\},\] with $\alpha=b+1$ if no such integer exists and let \[\beta = \min\{\alpha < i \leq b \mid \mu_i \neq 2e-2,e-1\},\] with $\beta=b+1$ if no such integer exists.  (We assume $\mu_{b+1}=0$.)  If $\sigma$ and $\tau$ are two compositions, define the composition $\sigma+\tau$ by $(\sigma+\tau)_i=\sigma_i+\tau_i$ for all $i$.    

\begin{proposition}
Suppose that $\la_1>\mu_1$, that $\mu$ has a good shape and that $N^\mu_{e-1}=0$.      
Set
\begin{align*}
\mu^{(1)}&=\mu^\ast+(\mu_\beta,\mu_\alpha), \\
\mu^{(2)}&=\mu^\ast+(\mu_\alpha-e+1,\mu_\beta+(e-1)).
\end{align*}
Then we have the following homomorphisms.  
\begin{itemize}
\item Suppose that $e-1<\mu_\beta$ or $\mu_\alpha<e-1$. Then $\EHom_{\h}(S^\mu,S^\la)=1$ if and only if $\la=\mu^{(1)}$.
\item Suppose that $0<\mu_\beta<e-1<\mu_\alpha$.  Then $\EHom_{\h}(S^\mu,S^\la)=1$ if and only $\la=\mu^{(1)}$ or $\la=\mu^{(2)}$.  
\item Suppose that $\mu_\beta=0$ and $e-1<\mu_\alpha$.  Then $\EHom_{\h}(S^\mu,S^\la)=1$ if and only $\la=\mu^{(2)}$.  
\end{itemize}
\end{proposition}

\begin{proposition}
Suppose that $\la_1>\mu_1$, that $\mu$ has a good shape and that $N^\mu_{e-1}\geq 1$.  Suppose $k,m \geq 0$ are such that
\begin{align*}
\mu_\beta+(N^\mu_{e-1}+1)(e-1)& \geq \mu_\alpha+ke, \\ \mu_\alpha+(N^\mu_{e-1}-1)(e-1) &\geq me.
\end{align*}
Set
\begin{align*}
\mu^{(1)}&=\mu^\ast+(\mu_\beta,\mu_\alpha+N^\mu_{e-1}(e-1)), \\
\mu^{(2)}&=\mu^\ast+(\mu_\alpha-e+1,\mu_\beta+(N^\mu_{e-1}+1)(e-1)), \\
\mu^{(3,m)}&=\mu^\ast+(\mu_\alpha+(N^\mu_{e-1}-1)(e-1)-me,me+e-1), \\
\mu^{(4,k)}&=\mu^\ast+(\mu_\beta+N^\mu_{e-1}(e-1)-ke,\mu_\alpha+ke).
\end{align*}
Then we have the following homomorphisms.  
\begin{itemize}
\item If $0=\mu_\beta \leq \mu_\alpha<e-1$ then $\EHom_{\h}(S^\mu,S^\la)=1$ if and only $\la = \mu^{(1)}$ or $\la=\mu^{(3,m)}$ for some $m$ as above.
\item If $\mu_\beta<e-1<\mu_\alpha$ then $\EHom_{\h}(S^\mu,S^\la)=1$ if and only $\la=\mu^{(2)}$ or $\la=\la^{(4,k)}$ for some $k$ as above.      
\end{itemize}    
\end{proposition}

\begin{proposition} \label{last}
Suppose that $\ell(\mu) \geq 4$, that $\la_1 > \mu_1$, that $\mu_3 \neq e-1$ and that $\mu$ does not have a good shape.  Then
$\dim(\EHom_{\h}(S^\mu,S^\la))=0$.
\end{proposition}

Combining Propositions~\ref{first} to~\ref{last} above completely classifies the homomorphism space $\EHom_{\h}(S^\mu,S^\la)$ where $\h=\h_{\C,q}(\sym_n)$, $\ell(\la)\leq 2$ and $\mu_1 \geq \la_2$.    

Our proof of these results is obtained via case-by-case analysis; we are doing nothing more than solving systems of homogeneous linear equations.  Unfortunately, there are many cases to check and the resulting computations are repetitive and formulaic.  We do not, therefore, propose to prove all the propositions in this paper.  In Section~\ref{SpaceProof}, we highlight the methods used and illustrate them with some examples.       

The computations that helped lead us to these results were carried out using GAP~\cite{GAP}.

\section{Proofs} \label{Proofs}
In this section, we give the proofs of Theorem~\ref{hdtthm} and Theorem~\ref{Lemma7} and indicate the proof of Propositions~\ref{first} to~\ref{last}.  For obvious reasons, this section is more technical than those preceeding it.  

\subsection{Proof of Theorem~\ref{hdtthm}} \label{hdtproof}
Let $\h=\h_{R,q}(\sym_n)$ and fix a partition $\mu=(\mu_1,\ldots,\mu_b) \vdash n$.  For $1 \leq d \leq b-1$ and $1 \leq t \leq \mu_{d+1}$, recall that 
\[h_{d,t} = \CC(\bar{\mu}_{d-1}; \mu_d,t),\]
and that $\mathcal{I}$ is the right ideal of $\h$ generated by
\[\{m_\mu h_{d,t} \mid 1 \leq d \leq b-1, 1 \leq t \leq \mu_{d+1}\}.\]
We prove Theorem~\ref{hdtthm}, that is, that
\[\mathcal{I} = M^\mu \cap \hmu.\]

Let $\mathcal{D}_\mu$ be a set of minimal length right coset representatives for $\sym_\mu$ in $\sym_n$ and recall~\cite[Propn.~3.3]{M:ULect} that $\mathcal{D}_\mu = \{ d \in \sym_n \mid \ft^\mu d \in \rowstd(\mu)\}$.  

\begin{lemma} [{~\cite[Cor.~3.4]{M:ULect}}] 
The module $M^\mu$ is a free $R$-module with basis
\[\{m_\mu T_{d(\ft)} \mid \ft \in \rowstd(\mu)\},\]
with the action of $\h$ determined by
\[m_\mu T_{d(\ft)} T_{i} = \begin{cases}
q m_{\mu} T_{d(\ft)}, & i,i+1 \text{ lie in the same row of }\ft, \\
m_\mu T_{d(\fs)}, & i\text{ lies above } i+1 \text{ in } \ft, \\
q m_\mu T_{d(\fs)} + (q-1)m_\mu T_{d(\ft)}, & i\text{ lies below } i+1 \text{ in } \ft, \\
\end{cases}\]
where $\fs=\ft(i,i+1)$.  

\end{lemma}
If $v \in \sym_\mu$ and $d \in \mathcal{D}_\mu$ then $\ell(vd)=\ell(v)+\ell(d)$ and $T_{vd}=T_vT_d$.    
Therefore if $w \in \sym_n$ then $w=vd$ for some $v \in \sym_{\mu}$ and $d \in \mathcal{D}_\mu$ so that $m_\mu T_w = q^{\ell(v)} m_{\mu}T_{d}$.  Equally, if $\ft^\mu w = \fs$, let $\dot\fs$ be the row-standard tableau obtained by rearranging the entries in each row of $\fs$.  Then $d=d(\fs)$ and $m_\mu T_w = q^{\ell(v)} m_\mu T_{d(\dot\fs)}$.    

If $\nu$ is a composition of $n$, define an equivalence relation $\sim_r$ on $\mathcal{T}(\nu,\mu)$ by saying that $\S\sim_r \T$ if $\S_j^i = \T_j^i$ for all $i,j$. 
Recall that if $\fs \in \mathcal{T}(\nu)$, then $\mu(\fs) \in \mathcal{T}(\nu,\mu)$ is the tableau obtained by replacing each integer $i\geq 1$ with its row index in $\ft^\mu$. 
Now if $\S \in \mathcal{T}(\nu,\mu)$ and $\ft \in \rowstd(\nu)$, define 
\[m_{\S\ft} = \sum_{{\fs \in \rowstd(\nu) \atop \mu(\fs)=\S}} m_{\fs\ft}.\]

\begin{lemma} [{~\cite[Thm.~4.9]{M:ULect}}] \label{mstbasis}
The right ideal $M^\mu \cap \hmu$ has a basis 
\[\{m_{\S\ft} \mid \S \in \mathcal{T}_0(\nu,\mu), \ft \in \Std(\nu) \text{ for some } \nu \vdash n \text{ such that } \nu \rhd \mu\}.\]  
\end{lemma}

\begin{lemma} [{~\cite[Eqn.~4.6]{M:ULect}}] \label{leftright}
Suppose $\nu$ is a composition of $n$ and $\S \in \mathcal{T}(\nu,\mu)$.  Then
\[m_{\S\ft^\nu}=\sum_{\S' \sim_r \S} m_\mu T_{\S'}.\]
\end{lemma}

\begin{ex}
Let $\nu=(3,2)$ and $\mu=(2,2,1)$.  Let $\S=\tab(112,23)$.  Then
\[m_{\S\ft^{(3,2)}} = (1+T_3)m_{(3,2)} = m_{(2,2,1)}(I+T_2+T_2T_1)(I+T_4).\]
\end{ex}

\begin{lemma}[{~\cite[Lemma~3.10]{M:ULect}}]  \label{minihigh} 
Suppose $\nu$ is any composition of $n$.  Let $\la$ be the partition obtained by rearranging the parts of $\nu$.  If $\fs,\ft \in \rowstd(\nu)$ then $m_{\fs\ft}$ is an $R$-linear combination of elements of the form $m_{\fu\fv}$ where $\fu$ and $\fv$ are row-standard $\la$-tableaux.   
\end{lemma}

\begin{lemma}\label{higher}
Suppose $\nu$ is any composition of $n$ such that $\nu \rhd \mu$ and $\fu,\fv \in \rowstd(\la)$. Then $m_{\fu\fv} \in \hmu$.
\end{lemma}

\begin{proof}
This follows from Lemma~\ref{minihigh}, noting that $\la \unrhd \nu \rhd \mu$.  
\end{proof}

Now suppose $1 \leq d < b$ and $1 \leq t \leq \mu_{d+1}$ and let $\nu=\nu(d,t)$ be the composition given by
\[\nu_i=\begin{cases}
\mu_{i}+t, & i=d, \\
\mu_{i}-t, & i=d+1, \\
\mu_{i}, & \text{otherwise}.
\end{cases}\]
Let $\S=\S_{d,t}$ be the the row-standard $\nu$-tableau such that $\S^{d}_d = \mu_d$, $\S^{d+1}_d = t$ and $\S^i_i = \nu_i$ for all $i \neq d$.  By Lemma~\ref{leftright}, $m_\mu h_{d,t} = m_{\S\ft^\nu} \in \hmu$.  The next result then follows by Lemma~\ref{higher}, noting that $M^\mu \cap \hmu$ is a right ideal of $\h$.  

\begin{corollary} \label{halfsub} We have
\[\mathcal{I} \subseteq M^\mu \cap \hmu.\]
\end{corollary}

Now we introduce some new notation which will help us describe the elements $m_{\S\ft}$.  If $\S\in \RowT(\nu,\mu)$, let $\ft_\S \in \rowstd(\mu)$ be the row-standard $\mu$-tableau in which $i$ is in row $r$ if the place occupied by $i$ in $\ft^\nu$ is occupied by $r$ in $\S$.  If $\ft_\S = \ft^\mu w$ then define $T_\S=T_w$.   
 
\begin{corollary} \label{DescribeS}
Suppose $\nu \vdash n$ and $\S \in \RowT(\nu,\mu)$.  Then
\[m_{\S\ft^\nu} = m_\mu T_\S \prod_{i\geq 1} \CC(\bar{\nu}_{i-1}; \S^1_i,\ldots,\S^b_i).\]
\end{corollary}

Recall that the length $\ell(w)$ of a permutation $w \in \sym_n$ may be determined by
\begin{align} \label{length} \ell(w)=\#\{ (i,j) \mid 1 \leq i <j \leq n \text{ and } i w > j w\},\end{align}
and if $w=uv \in \sym_n$ is such that $\ell(w)=\ell(u)+\ell(v)$ then $T_w=T_uT_v$. If $s \geq r \geq 1$ and $x \geq 1$ define 
\begin{align*}
\Per(s,r)&=(r,r+1,\ldots,s), \\
\D(s,r)&= T_{\Per(s,r)} = T_{s-1}T_{s-2}\ldots T_r, \\
\intertext{and}
\Per^\flat(s,r,x) &= \prod_{j=1}^x \Per(s+j,r+j), \\
\Dt(s,r,x) &= T_{\Per^\flat(s,r,x)}=\prod_{j=1}^x \D(s+j,r+j).
\end{align*}
Note that the last identity holds since $\ell({\Per^\flat(s,r,x)}) = x(s-r)$. Observe that if $s \geq t \geq r$ then
\[\Dt(s,r,x) = \Dt(s,t,x)\Dt(t,r,x).\]

\begin{lemma} \label{GetOrder}
Suppose $\nu \vdash n$ and $\S \in \RowT(\nu,\mu)$. 
Write $\fs=\ft_{\S}$. Let $\fs(0)=\ft^\mu$ and, for $1 \leq i \leq n$, let 
\[\fs(i) = \fs(i-1)\Per(i^\ast,i)\]
where $i^\ast$ occupies the same position in $\fs(i-1)$ that $i$ occupies in $\fs$.  Then $\fs(i)$ is the row-standard $\mu$-tableau with the entries $1,\ldots,i$ occupying the same positions that they occupy in $\fs$ and with all other entries in row order, and furthermore 
\begin{equation} \label{Eqni}
T_{d(\fs(i))} = \prod_{j=1}^i \D(j^\ast,j).\end{equation}  
In particular,
\[m_{\mu}T_\S = m_\mu \prod_{j=1}^n \D(j^\ast,j).\]
\end{lemma}

\begin{proof}
The description of $\fs(i)$ is easily seen by induction on $i$.  Equation~\ref{Eqni} follows, using induction or otherwise, by observing that
\[\ell(d(\fs(i))) = \sum_{j=1}^i \ell(\Per(j^\ast,j)).\]  
\end{proof}

For $m,a,b \geq 0$ define 
\[\seq{m}{a}{b}=\{{\bf i}=(i_1,\ldots,i_b) \mid m+1 \leq i_1<\ldots<i_b\leq m+a+b\}.\]

\begin{lemma}
Let $m \geq 0$ and let $(a,b)$ be a composition.  Then   
\[\CC(m;a,b)= \sum_{{\bf i} \in \seq{m}{a}{b}} \prod_{k=1}^{b} \D(m+a+k,i_k).\]
\end{lemma}

\begin{proof}
We may assume $m=0$.  Let ${\bf i} \in \seq{0}{a}{b}$.  If $w = \prod_{k=1}^{b}\Per(a+k,i_k)$ then by Equation~\ref{length}, 
\[\ell(w)= \sum_{k=1}^{b}(a+k-i_k)=\sum_{k=1}^b \ell(\pi(a+k,i_k)),\] 
so that $T_w = \prod_{k=1}^{b}\D(a+k,i_k)$.  Now recall that $w \in \mathcal{D}_{0,(a,b)}$ if and only if $\ft^{(a,b)} w \in \rowstd((a,b))$ and observe that $\ft^{(a,b)} w$ is precisely the row-standard $(a,b)$-tableau with the entries $i_1,\ldots,i_{b}$ in the second row.
\end{proof}

\begin{corollary} \label{RearrangeCor} Suppose that $m,a,b \geq 0$ and $\bar{m} \geq m+a$.  Then
\[\sum_{{\bf i} \in \seq{m}{a}{b}} \prod_{k=1}^b \D(\bar{m}+k,i_k) = \Dt(\bar{m},m+a,b) \CC(m;a,b).\] 
\end{corollary}

\begin{proof}
\begin{align*}
\sum_{{\bf i} \in \seq{m}{a}{b}} \prod_{k=1}^b \D(\bar{m}+k,i_k) & 
= \sum_{{\bf i} \in \seq{m}{a}{b}} \prod_{k=1}^b \D(\bar{m}+k,m+a+k)  \D(m+a+k,i_k) \\
&= \prod_{k=1}^b \D(\bar{m}+k,m+a+k) \times \sum_{{\bf i} \in \seq{m}{a}{b}} \prod_{k=1}^b \D(m+a+k,i_k) \\
& = \Dt(\bar{m},m+a,b) \CC(m;a,b)
\end{align*}
\end{proof}

\begin{lemma} \label{PullsThrough}
Suppose $m,a,b \geq 0$ and $\bar{m} \geq m$.  Then
\[\CC(\bar{m};a,b)\Dt(\bar{m},m,a+b) = \Dt(\bar{m},m,a+b)\CC(m;a,b).\]
\end{lemma}

\begin{proof}
If ${\bf i} =(i_1,\ldots,i_b) \in \seq{m}{a}{b}$, define ${\bf \bar{i}} = (\bar{i}_1,\ldots,\bar{i}_k) \in \seq{\bar{m}}{a}{b}$ by setting $\bar{i}_k = i_k+\bar{m}-m$ for all $k$.  We claim that 
\[\prod_{k=1}^b \D(\bar{m}+a+k,\bar{i}_k) \times \Dt(\bar{m},m,a+b) = \Dt(\bar{m},m,a+b) \times \prod_{k=1}^b \D(m+a+k,i_k)\]  
for all ${\bf i} \in \seq{m}{a}{b}$.  Let $\eta=(\bar{m},a,b)$ and let $\ft \in \rowstd(\eta)$ be the tableau containing $1,\ldots,m,m+a+b+1,\ldots,\bar{m}+a+b$ in the first row and $i_1,\ldots,i_b$ in the third row.  Let $w$ be the permutation such that $\ft^\eta w = \ft$.  It is sufficient to check that
\[\prod_{k=1}^b \Per(\bar{m}+a+k,\bar{i}_k) \times \Per^\flat(\bar{m},m,a+b) = w=\Per^\flat(\bar{m},m,a+b) \times \prod_{k=1}^b \Per(m+a+k,i_k)\]
and that \[\ell(w) = (\bar{m}-m)(a+b) + \sum_{k=1}^b (m+a+k-i_k),\]
which is a routine exercise.   
\end{proof}

\begin{lemma} \label{PullsThrough2}
Suppose $m,a,b,r \geq 0$.  Then
\[\CC(m;a,b)\Dt(m+a+b,m,r) = \Dt(m+a+b,m,r)\CC(m+r;a,b).\]
\end{lemma}

\begin{proof}
The proof is similar to the proof of Lemma~\ref{PullsThrough}; we consider tableaux of shape $\eta=(m,a,b,r)$.  
\end{proof}

\begin{lemma} \label{IntoThree}
Let $m\geq 0$ and $\eta=(\eta_1,\eta_2,\ldots,\eta_l)$ be a composition. Suppose  $0 \leq x\leq l$. Then
\[\CC(m;\eta_1,\eta_2,\ldots,\eta_l) = \CC(m;\eta_1,\ldots,\eta_x) \CC(m+\bar{\eta}_x; \eta_{x+1},\ldots,\eta_l) \CC(m; \bar{\eta}_x, \bar{\eta}_l-\bar{\eta}_x).\]
 
\end{lemma}

\begin{proof}
Again, we may assume that $m =0$.  Let $\ft \in \rowstd(\eta)$ and let $w$ be the permutation such that $\ft = \ft^\eta w$.  Suppose that the entries in rows $1,\ldots,x$ of $\ft$ are $j_1<\ldots<j_{\bar\eta_x}$ and that the entries in rows $x+1,\ldots,l$ are $i_{\bar\eta_x+1}<\ldots<i_{\bar\eta_l}$.  Let $w_1$ be the permutation of $\{1,\ldots,\bar\eta_x\}$ which sends $1 \leq \alpha \leq \bar\eta_x$ to $\alpha^\ast$, where $j_{\alpha^\ast}$ occupies the same position in $\ft$ that $\alpha$ occupies in $\ft^\eta$.  Similarly, let $w_2$ be the permutation of $\{\bar{\eta}_x+1,\ldots,\bar\eta_l\}$ which sends $\bar{\eta}_x+1 \leq \alpha \leq \bar\eta_x$ to $\alpha^\ast$, where $i_{\alpha^\ast}$ occupies the same position in $\ft$ that $\alpha$ occupies in $\ft^\eta$.      
Finally let $w_3 = \prod_{k=\bar\eta_x+1}^{\bar\eta_l} \Per(k,i_k)$.  It is clear that $w_1 \in C(0;\eta_1,\ldots,\eta_x), \, w_2 \in C(\bar\eta_{x};\eta_{x+1},\ldots,\eta_l)$ and $w_3 \in C(0; \bar\eta_x,\bar\eta_l-\bar\eta_x$.  
We leave it as an exercise to check that $w=w_1w_2w_3$ and that $\ell(w)=\ell(w_1)+\ell(w_2)+\ell(w_3)$.  
The proof of Lemma \ref{IntoThree} follows by counting the number of terms on both sides of the equation.       
\end{proof}

\begin{lemma} \label{termstofront}
Suppose $\nu$ is a partition of $n$ with $\nu \rhd \mu$ and $\S\in \mathcal{T}_0(\nu,\mu)$.  Choose $k$ minimal such that $\nu_k>\mu_k$ and $r>k$ minimal such that $\S^r_k>0$.  Then  
\[m_{\S \ft^\nu} = m_\mu \Dt(\bar\mu_{r-1},\bar\mu_k,\S^r_k)\CC(\bar\mu_{k-1}; \mu_k,\S^r_k) h\]
for some $h \in \h$.
\end{lemma}

\begin{proof} Using Lemma \ref{DescribeS} and Lemma \ref{IntoThree},
\begin{align*}
m_{\S \ft^\nu} &  = m_\mu T_\S \prod_{l\geq k}\CC\left(\bar\nu_{l-1}; \S^l_l,\ldots \S^b_l \right) \\
&=m_\mu T_\S \CC(\bar\mu_{k-1}; \mu_k,\S^r_k,\ldots,\S^{b}_k)\prod_{l>k }\CC\left(\bar\nu_{l-1}; \S^l_l,\ldots, \S^b_l\right) \\
&=m_\mu T_\S \CC(\bar\mu_{k-1}; \mu_k,\S^r_k)\CC(\bar\mu_{k}+\S^r_k;\S^{r+1}_k,\ldots,\S^b_k)\CC(\bar\mu_{k-1};\mu_k+\S^r_k,\nu_k-\mu_k-\S^r_k)\\
& \hspace{10mm}\times \prod_{l>k}\CC\left(\bar\nu_{l-1}; \S^l_l,\ldots,\S^{b}_l\right).
\end{align*}
As in Lemma \ref{GetOrder}, and keeping the notation of that lemma, we may write
\[m_\mu T_\S = m_\mu \prod_{l = \bar{\mu}_{k}+1}^n \D(l^\ast,l) = m_\mu \Dt(\bar\mu_{r-1},\bar\mu_k,\S^r_k)\check h\] where 
\[\check h = \prod_{l=\bar\mu_k+\S^r_k+1}^n \D(l^\ast,l) \]
commutes with $\CC(\bar\mu_{k-1};\mu_k,\S^r_k)$. The result follows.  
\end{proof}

\begin{lemma}\label{killsall}
Suppose that $1\leq k< r \leq b$ and $1 \leq x \leq \mu_{r}$.  Then 
\[m_\mu \Dt(\bar\mu_{r-1},\bar\mu_{k},x) \CC(\bar\mu_{k-1};\mu_{k},x) \in \mathcal{I}.\]
\end{lemma}

\begin{proof}
It is straightforward to see that the proof for arbitrary $k$ is identical to the proof for $k=1$, so we assume that $k=1$.  
We now prove that the lemma holds for all $2 \leq r \leq b$.  
If $r=2$ and $1 \leq x \leq \mu_2$ then
\[m_\mu \CC(0;\mu_1,x) = m_\mu h_{1,x} \in \mathcal{I}.\]
So now suppose that $3 \leq r \leq b$ and that Lemma \ref{killsall} holds for all $r' < r$.  Choose $x$ with $1 \leq x \leq \mu_r$.  Recall that if $1 \leq t \leq \mu_{r}$ then  
\[h_{r-1,t} = \sum_{{\bf i} \in \seq{\bar\mu_{r-2}}{\mu_{r-1}}{t}} \left(\prod_{l=1}^t \D(\bar\mu_{r-1}+l,i_l)\right).\]
We have
\[m_\mu \Dt(\bar\mu_{r-1},\mu_1,x) \CC(0;\mu_{1},x) = m_\mu \Dt(\bar\mu_{r-1},\bar\mu_{r-2},x) \Dt(\bar\mu_{r-2},\mu_1,x) \CC(0;\mu_1,x) \]
where
\[\Dt(\bar\mu_{r-1},\bar\mu_{r-2},x) = h_{r-1,x}-\sum_{j=1}^{\mu_{r-1}} \sum_{{\bf i} \in \seq{\bar\mu_{r-2}}{j}{x-1}} \prod_{l=1}^{x-1} \Big( \D(\bar\mu_{r-1}+l,i_l) \Big) \D(\bar\mu_{r-1}+x,\bar\mu_{r-2}+x+j) \]
so that it is sufficient to show that
\begin{equation} \label{toshow} m_\mu \sum_{j=1}^{\mu_{r-1}} \sum_{{\bf i} \in \seq{\bar\mu_{r-2}}{j}{x-1}} \prod_{l=1}^{x-1} \Big( \D(\bar\mu_{r-1}+l,i_l)\Big) \D(\bar\mu_{r-1}+x,\bar\mu_{r-2}+x+j) \Dt(\bar\mu_{r-2},\mu_1,x) \CC(0;\mu_1,x) \in \mathcal{I}. 
\end{equation}
Now for $1 \leq j \leq \mu_{r-1}$ we may write
\begin{multline*} \sum_{{\bf i} \in \seq{\bar\mu_{r-2}}{j}{x-1}} \prod_{l=1}^{x-1} \D(\bar\mu_{r-1}+l,i_l)\\
= \sum_{y=\max\{0,x-j\}}^{x-1} \left(\Big(\sum_{{\bf i} \in \seq{\bar\mu_{r-2}}{x-y}{y}} \prod_{l=1}^y \D(\bar\mu_{r-1}+l,i_l)\Big)
\Big(\sum_{{\bf i} \in \seq{\bar\mu_{r-2}+x}{j-x+y}{x-y-1}} \prod_{l=1}^{x-y-1} \D(\bar\mu_{r-1}+y+l,i_l)\Big)
\right)
\end{multline*}
so that, substituting into Equation \ref{toshow} and commuting terms to the left where possible, we must show that 
\begin{multline*}
m_\mu \sum_{y=0}^{x-1} \sum_{j=x-y}^{\mu_{r-1}}\sum_{{\bf i} \in \seq{\bar\mu_{r-2}}{x-y}{y}} 
\left(\prod_{l=1}^y \D(\bar\mu_{r-1}+l,i_l) \right) \Dt(\bar\mu_{r-2},\mu_1,x) 
\CC(0;\mu_1,x) \\
\times \sum_{{\bf i} \in \seq{\bar\mu_{r-2}+x}{j-x+y}{x-y-1}} \left(\prod_{l=1}^{x-y-1} \D(\bar\mu_{r-1}+y+l,i_l) \right)\D(\mu_{r-1}+x,\bar\mu_{r-2}+x+j) \in \mathcal{I}.
\end{multline*}
Consider $y=0$.  By the inductive hypothesis,
\[m_\mu \Dt(\bar\mu_{r-2},\mu_1,x) \CC(0;\mu_1,x) \in \mathcal{I}\]
so that it is sufficient to prove that for $1 \leq y  <x$ we have 
\[m_\mu \sum_{{\bf i} \in \seq{\bar\mu_{r-2}}{x-y}{y}} \left(\prod_{l=1}^y \D(\bar\mu_{r-1}+l,i_l)\right) \Dt(\bar\mu_{r-2},\mu_1,x) \CC(0;\mu_1,x)\in \mathcal{I}.\] 
Now, using Corollary \ref{RearrangeCor}, Lemma \ref{PullsThrough} and Lemma \ref{IntoThree},
\begin{align*} 
\sum_{{\bf i} \in \seq{\bar\mu_{r-2}}{x-y}{y}}
\prod_{l=1}^y &\D(\bar\mu_{r-1}+l,i_l)\Dt(\bar\mu_{r-2},\mu_1,x) \CC(0;\mu_1,x) \\
& = \Dt(\bar\mu_{r-1},\bar\mu_{r-2}+x-y,y) \CC(\bar\mu_{r-2}; x-y,y) \Dt(\bar\mu_{r-2},\mu_1,x) \CC(0;\mu_1,x) \\
& = \Dt(\bar\mu_{r-1},\bar\mu_{r-2}+x-y,y) \Dt(\bar\mu_{r-2},\mu_1,x) \CC(\mu_1;x-y,y)\CC(0;\mu_1,x) \\
&= \Dt(\bar\mu_{r-1},\bar\mu_{r-2}+x-y,y) \Dt(\bar\mu_{r-2},\mu_1,x)\CC(0;\mu_1,x-y,y) \\
& = \Dt(\bar\mu_{r-1},\bar\mu_{r-2}+x-y,y) \Dt(\bar\mu_{r-2},\mu_1,x) \CC(0;\mu_1,x-y)\CC(0;\mu_1+x-y,y) \\
& = \Dt(\bar\mu_{r-1},\bar\mu_{r-2}+x-y,y) \Dt(\bar\mu_{r-2},\mu_1,x-y) \Dt(\bar\mu_{r-2}+x-y,\mu_1+x-y,y) \\ 
& \hspace{1cm} \CC(0;\mu_1,x-y)\CC(0;\mu_1+x-y,y) \\
&= \Dt(\bar\mu_{r-2},\mu_1,x-y) \CC(0;\mu_1,x-y) \Dt(\bar\mu_{r-1}, \bar\mu_{r-2}+x-y,y) \\ 
& \hspace{1cm} \Dt(\bar\mu_{r-2}+x-y,\mu_1+x-y,y) \CC(0;\mu_1+x-y,y)
\end{align*}
and by the inductive hypothesis again,
\[m_\mu  \Dt(\bar\mu_{r-2},\mu_1,x-y)\CC(0;\bar\mu_1,x-y)\in \mathcal{I}.\]
\end{proof}

\begin{proposition} \label{AllEqual}
We have
\[M^\mu \cap \hmu = \mathcal{I}.\]
\end{proposition}  

\begin{proof}
By Corollary \ref{halfsub}, $\mathcal{I} \subseteq M^\mu \cap \hmu$.  Now by Lemma \ref{mstbasis}, $M^\mu  \cap \hmu$ has a basis 
\[\{m_{\S \ft} \mid \S \in \mathcal{T}_0(\nu,\mu),\ft \in \Std(\nu) \text{ for some } \nu \rhd \mu\}.\]  If $\nu \rhd \mu$ and $\S \in \mathcal{T}_0(\nu,\mu), \ft \in \Std(\nu)$ then, since $\mathcal{I}$ is a right ideal, it follows from Lemmas \ref{termstofront} and \ref{killsall} that $m_{\S \ft^\nu} \in \mathcal{I}$ and so $m_{\S\ft} \in \mathcal{I}$.  Hence $M^{\mu} \cap \hmu \subseteq \mathcal{I}$.  
\end{proof}

\subsection{Proof of Theorem~\ref{Lemma7}} \label{Lemma7Proof}
In this section, we give the proof of Theorem~\ref{Lemma7}.  Let $\q$ be an indeterminate over $\Z$ and let $\calZ=\Z[\q,\q^{-1}]$.  Let $\h^{\calZ}=\h_{\calZ,\q}(\sym_n)$.  We prove Theorem~\ref{Lemma7} for $\h=\h^{\calZ}$; the general result follows by specialization.  

Let $\la=(\la_1,\ldots,\la_a)$ be a partition of $n$, $\nu=(\nu_1,\ldots,\nu_b)$ a composition of $n$ and $\S \in \RowT(\la,\nu)$ where we assume that $a \geq 2$.  Our aim is to write $\Theta_\S:M^\nu \rightarrow S^\la$ as a linear combination of homomorphisms indexed by tableaux $\U \in \RowT(\la,\nu)$.  As in the previous examples,  we identify $\U \in \RowT(\la,\nu)$ with $\Theta_{\U}(m_\nu)$.  

\begin{ex} \label{NiceEx}
Let $\la=(3,3)$ and $\nu=(2,1,1,1,1)$.  
Recall that \[m_\la h_{1,1} = m_\la(I+T_3+T_3T_2+T_3T_2T_1) \in \hla.\]
Then
\begin{align*}
\tab(114,235) &= \hla+m_{\la}T_3T_4 \\
&= \hla + \q^{-1}m_{\la}T_4T_3T_4 \\
&= \hla + \q^{-1}m_\la T_3T_4T_3 \\
&= \hla - \q^{-1}m_\la(I+T_3T_2+T_3T_2T_1)T_4T_3 \\
&=\hla - m_\la (T_3 +\q^{-1}T_3T_2T_4T_3 + \q^{-1}T_3T_2T_1T_4T_3) \\
&= - \tab(113,245) -\q^{-1}\tab(134,125). 
\end{align*}
\end{ex}

If $\U \in \RowT(\la,\nu)$, define $\first(\U) \in \rowstd(\la)$ to be the tableau with $\nu(\first(\U))=\U$ and $\{\nu_i+1,\ldots,\nu_{i+1}\}$ in row order, for all $0 \leq i < b$.  The following result is given by the definition of the map $\Theta_\U$.  (Observe also the `reverse' statement in Corollary~\ref{DescribeS}.)  

\begin{lemma} \label{reverse} Let $\U \in \RowT(\la,\nu)$.  Then
\[\Theta_{\U}(m_\nu) = \hla+ m_\la T_{d(\first(\U))} \prod_{j=1}^b C(\bar\nu_{j-1}; \U^j_1,\ldots,\U^j_a).\]
\end{lemma}

We begin by considering the case where $\la=(\la_1,\la_2)$.  If $\S \in \RowT(\la,\nu)$ is such that $\S^i_1=\alpha_i$  and $\S^i_2=\beta_i$ for $1 \leq i \leq b$ then we will represent $\S$ by
\[\S = \rep{1^{\alpha_1}2^{\alpha_2}\ldots b^{\alpha_b}}{1^{\beta_1} 2^{\beta_2} \ldots b^{\beta_b}}.\]
If a number $i$ does not appear in the top (resp. bottom) row it should be understood that $\alpha_i=0$ (resp. $\beta_i=0$).  Any such representation containing an index $\alpha_i,\beta_i<0$ should be taken to be zero, so for example in Corollary~\ref{basic5} below, if $\alpha_1=0$ then the first term on the right-hand side of the equation should be ignored.  We continue to identify $\S$ with $\Theta_\S(m_\nu)$.  

\begin{lemma} \label{Type1n} 
Let $0 \leq \alpha \leq \la_1$ and $0 \leq \beta < \la_2$ and let $\S \in \rowstd(\la)$ be the tableau with $1,\ldots,\alpha,\alpha+\beta+2,\ldots,\la_1+\beta+1$ in the top row (and $\alpha+1,\ldots,\alpha+\beta+1,\la_1+\beta+2,\ldots,\la_1+\la_2$ in the second row).  For $i \in \{1,\ldots,\alpha,\alpha+\beta+2,\ldots,\la_1+\beta+1\}$ let $\U_i$ be the row-standard tableau obtained from $\S$ by swapping $\alpha+\beta+1$ with $i$ and rearranging the rows if necessary.  Then  
\[\Theta_\S = - \q^{-\beta}\sum_{i=1}^\alpha \Theta_{\U_i} - \sum_{i=\alpha+\beta+2}^{\la_1+\beta+1}\Theta_{\U_i}.\] 
\end{lemma}

\begin{proof} We use Lemma~\ref{PullsThrough2} and note that $m_\la h_{1,1} =m_\la C(0;\la_1,1) \in \hla$.  If $\alpha =0$ then
\begin{align*}
\Theta_{\S}(m_\nu) + \sum_{i= \beta+2}^{\la_1+\beta+1} \Theta_{\U_i}(m_\nu) 
&= \hla + m_\la \Dt(\la_1,0,\beta) \CC(\beta; \la_1,1) \\
&= \hla + q^{-\beta}  m_\la \Dt(\la_1+1,0,\beta) \CC(\beta; \la_1,1) \\
&=\hla + q^{-\beta} m_\la \CC(0; \la_1,1) \Dt(\la_1+1,0,\beta) \\
&=\hla \\
&=0. 
\end{align*}
Else if $\alpha >0$ then  
\begin{align*}
\sum_{i=1}^\alpha \Theta_{\U_i}(m_\nu) &= \hla + m_\la \Big(\sum_{i=1}^\alpha \D(\la_1+1,i)\Big) \Dt(\la_1+1,\alpha,\beta) \\
&  =\hla+ m_\la \Big(h_{1,1}-\sum_{i=\alpha+1}^{\la_1+1} \D(\la_1+1,i)\Big) \Dt(\la_1+1,\alpha,\beta) \\ 
&= \hla- m_\la \CC(\alpha;\la_1-\alpha,1) \Dt(\la_1+1,\alpha,\beta) \\
&= \hla- m_\la \Dt(\la_1+1,\alpha,\beta)\CC(\alpha+\beta; \la_1-\alpha,1) \\
&= \hla- \q^{\beta} m_\la  \Dt(\la_1,\alpha,\beta) \CC(\alpha+\beta;\la_1-\alpha,1) \\
&=  -\q^\beta \Theta_\S(m_\nu) - \q^\beta \sum_{i=\alpha+\beta+2}^{\la_1+\beta+1}\Theta_{\U_i}(m_\nu).
\end{align*}
\end{proof}

Using the definition of the maps $\Theta_\U$, we have the following corollary.  
\begin{corollary} \label{basic5}
We have that
\[\rep{1^{\alpha_1} 4^{\alpha_4}}{2^{\beta_2} 3^1 5^{\beta_5}} = 
-\q^{-\beta_{2}} 
\rep{1^{\alpha_1-1} 3^{1} 4^{\alpha_4}}{1^1 2^{\beta_{2}} 5^{\beta_5}} 
- \rep{1^{\alpha_1} 3^{1} 4^{\alpha_4-1}}{2^{\beta_2} 4^1 5^{\beta_5}}.
\]
\end{corollary}

\begin{lemma} \label{cosetattack}
Suppose $1 \leq d \leq b$ is such that $\alpha_d=0$ and $\beta_d=1$.  Then 
\begin{multline*}\rep{1^{\alpha_1}2^{\alpha_2}\ldots d^0\ldots b^{\alpha_b}}{1^{\beta_1}2^{\beta_2}\ldots d^{1} \ldots b^{\beta_b}} 
= - \sum_{i=1}^{d-1} \q^{-\bar\beta_{d-1}+\bar\beta_{i-1}}[\beta_i+1] \rep{1^{\alpha_1}2^{\alpha_2} \ldots i^{\alpha_i-1} \ldots d^{1}  \ldots b^{\alpha_b}}{1^{\beta_1}2^{\beta_2}\ldots i^{\beta_i+1} \ldots d^0 \ldots b^{\beta_b}}\\
- \sum_{i=d+1}^{b} \q^{-\bar\beta_{d}+\bar\beta_{i-1}}[\beta_i+1]  \rep{1^{\alpha_1}2^{\alpha_2}\ldots d^{1} \ldots i^{\alpha_i-1} \ldots b^{\alpha_b}}{1^{\beta_1}2^{\beta_2}\ldots d^0 \ldots i^{\beta_i+1} \ldots b^{\beta_b}}.
\end{multline*}
\end{lemma}

\begin{proof}
Note that
\[\rep{1^{\alpha_1} 2^{\alpha_2}\ldots d^{0} \ldots b^{\alpha_b}}{1^{\beta_1}2^{\beta_2}\ldots d^{1} \ldots b^{\beta_b}} 
= \rep{1^{\bar\alpha_{d-1}}4^{\bar\alpha_b-\bar\alpha_d}}{2^{\bar\beta_{d-1}}3^{1}5^{\bar{\beta_b}-\bar\beta_d}}  T_w \prod_{i=1}^b \CC(\bar\alpha_{i-1}+\bar\beta_{i-1};\alpha_i,\beta_i) \]
where 
\[T_w=\prod_{i=1}^{d-1} \Dt(\bar\alpha_{d-1}+\bar\beta_{i-1},\bar\alpha_i+\bar\beta_{i-1},\beta_i)\prod_{i=d+1}^{b} \Dt(\bar\alpha_{b}+\bar\beta_{i-1},\bar\alpha_i+\bar\beta_{i-1},\beta_i) .\]
Let \[\V= \rep{1^{\bar\alpha_{d-1}}4^{\bar\alpha_b-\bar\alpha_d}}{2^{\bar\beta_{d-1}}3^{1}5^{\bar{\beta_b}-\bar\beta_d}} \]
and let $\ft \in \rowstd(\la)$ be the unique tableau such that $\nu(\ft)=\V$.  Choose $j$ in the top row of $\ft$ and let $\ft(j) \in \rowstd(\la)$ be the tableau obtained by swapping $j$ and $\bar\alpha_{d-1}+\bar\beta_{d-1}+1$ and rearranging the rows.  By Lemma~\ref{Type1n},
\[\Theta_\ft=-\q^{-\bar\beta_{d-1}}\sum_{j=1}^{\bar\alpha_{d-1}} \Theta_{\ft(j)} - \sum_{j=\bar\alpha_{d}+\bar\beta_{d}+1}^{\bar\alpha_b+\bar\beta_d}\Theta_{\ft(j)},\]
that is
\begin{align*}
\rep{1^{\alpha_1} 2^{\alpha_2}\ldots d^{0} \ldots b^{\alpha_b}}{1^{\beta_1}2^{\beta_2}\ldots d^{1} \ldots b^{\beta_b}} 
&= \hla -\Big(\q^{-\bar\beta_{d-1}}\sum_{j=1}^{\bar\alpha_{d-1}}m_\la T_{d(\ft(j))} + \sum_{j=\bar\alpha_{d}+\bar\beta_d+1}^{\bar\alpha_b+\bar\beta_d}m_\la T_{d(\ft(j))}\Big)\T_w \prod_{k=1}^b \CC(\bar\nu_{k-1};\alpha_k,\beta_k)\\
&= \hla -\Big(\q^{-\bar\beta_{d-1}} \sum_{i=1}^{d-1} \sum_{j=\bar\alpha_{i-1}+1}^{\bar\alpha_{i}}m_\la T_{d(\ft(j))} + \sum_{i=d+1}^b \sum_{j=\bar\beta_d+\bar\alpha_{i-1}+1}^{\bar\alpha_i+\bar\beta_d}m_\la T_{d(\ft(j))}\Big) \\
& \hspace*{15mm} \T_w \prod_{k=1}^b \CC(\bar\nu_{k-1};\alpha_k,\beta_k).
\end{align*}
Now choose $i$ with $1 \leq i \leq d-1$.  We claim that 
\[\sum_{j=\bar\alpha_{i-1}+1}^{\bar\alpha_i} m_\la T_{d(\ft(j))} T_w \prod_{k=1}^b \CC(\bar\nu_{k-1};\alpha_k,\beta_k)  = 
\q^{\bar\beta_{i-1}}[\beta_i+1]
\rep{1^{\alpha_1}2^{\alpha_2} \ldots i^{\alpha_i-1} \ldots d^{1}  \ldots b^{\alpha_b}}{1^{\beta_1}2^{\beta_2}\ldots i^{\beta_i+1} \ldots d^0 \ldots b^{\beta_b}}.\]
To prove the claim, choose $j$ with $\bar\alpha_{i-1}+1 \leq j \leq \bar\alpha_i$.  Let $\fs$ be the row-standard $\la$-tableau containing entries from the set  $ \cup_{i=1}^d 
\{ \bar\alpha_i +\bar\beta_{i-1} +1, \ldots, \bar\alpha_i+\bar\beta_i \}
$.  
Then 
\[\ft^\la d(\ft(j)) w  = \fs w'\]
where $w'$ acts on $\fs$ as follows.  Suppose that $x$ occupies the same position in $\fs$ that $j$ occupies in $\ft$.  Then $w'$ moves $\bar\nu_{d-1}+1$ into the first row, such that the row is in increasing order, moves $x$ into the far left of the second row and pushes the entries in the second row, up to the entry immediately to the left of $\bar\nu_{d-1}+1$ in $\fs$, one box to the right.  If $\dot\fs$ is the row-standard tableau obtained by rearranging the rows of $\fs w'$, then $d(\fs)w'=u d(\dot\fs)$ where $u \in \sym_\la$ has length $\bar\beta_{i-1}$.  
Furthermore, using Equation~\ref{length}, we may check that $\ell(d(\ft(j)))+\ell(w) = \ell(d(\fs)w')=\bar\beta_{i-1}+\ell(d(\dot\fs))$.  It therefore follows that 
\[\sum_{j=\bar\alpha_{i-1}+1}^{\bar\alpha_i} m_\la T_{d(\ft(j))} T_w = \q^{\bar\beta_{i-1}} \sum_{x=\bar\nu_{i-1}+1}^{\alpha_i+\beta_{i-1}} m_\la T_{d(\dot\fs(x))}\]
where $\dot\fs(x)$ is the tableau obtained from $\fs$ by swapping $x$ and $\bar\nu_{d-1}+1$ and rearranging the rows.
If we let $x_0 = \bar\nu_{i-1}+1$ then using Lemma~\ref{IntoThree},
\begin{align*}
\sum_{j=\bar\alpha_{i-1}+1}^{\bar\alpha_i} m_\la T_{d(\ft(j))} T_w C(\bar\nu_{i-1};\alpha_i,\beta_i) & = \q^{\bar\beta_{i-1}} m_\la T_{d(\dot\fs(x_0))} C(\bar\nu_{i-1};\alpha_i-1,1) \CC(\bar\nu_{i-1}; \alpha_i,\beta_i)\\     
&= \q^{\bar\beta_{i-1}} m_\la T_{d(\dot\fs(x_0))} C(\bar\nu_{i-1}; \alpha_i-1,1,\beta_i) \\
&= \q^{\bar\beta_{i-1}} m_\la T_{d(\dot\fs(x_0))} C(\bar\alpha_{i}+\bar\beta_{i-1}-1; 1, \beta_i)C(\bar\nu_{i-1};\alpha_{i}-1,\beta_i+1) \\
&= \q^{\bar\beta_{i-1}}[\beta_i+1] m_\la T_{d(\dot\fs(x_0))}  C(\bar\nu_{i-1};\alpha_{i}-1,\beta_i+1) \\
\end{align*}
The claim then follows since 
\[m_\la T_{d(\dot\fs(x_0))}  C(\bar\nu_{i-1};\alpha_{i}-1,\beta_i+1) \prod_{k\neq i} \CC(\bar\nu_{k-1};\alpha_k,\beta_k)=\rep{1^{\alpha_1}2^{\alpha_2} \ldots i^{\alpha_i-1} \ldots d^{1}  \ldots b^{\alpha_b}}{1^{\beta_1}2^{\beta_2}\ldots i^{\beta_i+1} \ldots d^0 \ldots b^{\beta_b}}.\]
A similar proof shows that if $d+1 \leq i \leq b$ then
\[\sum_{j=\bar\alpha_{i-1}+\beta_d+1}^{\bar\alpha_i+\bar\beta_d} \Theta_{\ft(j)}(m_\nu)T_w \prod_{k=1}^b \CC(\bar\nu_{k-1};\alpha_k,\beta_k) = 
\q^{-\bar\beta_{d}+\bar\beta_{i-1}}[\beta_i+1]
\rep{1^{\alpha_1}2^{\alpha_2}  \ldots d^{1} \ldots i^{\alpha_i-1} \ldots b^{\alpha_b}}{1^{\beta_1}2^{\beta_2}\ldots  d^0 \ldots i^{\beta_i+1}  \ldots b^{\beta_b}},\]
completing the proof of Lemma~\ref{cosetattack}.
\end{proof}

\begin{lemma} \label{Stage3}
Suppose 
\[\S=\rep{1^{\alpha_1} \ldots d^0 \ldots b^{\alpha_b}}{1^{\beta_2} \ldots d^{\beta_d} \ldots b^{\beta_b}}\]
where $\beta_d\geq 1$.  
Let \[\mathcal{G}=\{(g_1,\ldots,g_b) \mid g_d=0, \bar{g}=\beta_d \text{ and } g_i \leq \alpha_i \text{ for } 1 \leq i \leq b\}.\]  For $g \in \mathcal{G}$, let $\U_g$ be the row-standard tableau obtained from $\S$ by moving all entries equal to $d$ from row 2 to row 1, and for $i\neq d$ moving down $g_i$ entries equal to $i$ from row 1 to row 2.  Then  
\[\Theta_\S = \sum_{g \in \mathcal{G}} (-1)^{\beta_d} \q^{-\binom{\beta_d+1}{2}+\bar{g}_{d-1}} \q^{-\bar\beta_{d-1}\beta_d}\prod_{i=1}^b \q^{g_i(\bar\beta_{i-1})} \gauss{\beta_i+g_i}{g_i}\Theta_{\U_g}.\]
 \end{lemma}

\begin{proof} The case that $\beta_d=1$ is Lemma~\ref{cosetattack}.  So assume $\beta_d >1$ and that the lemma holds when $\S^d_2<\beta_d$.  We first consider
\[\S=\rep{1^{\alpha_1} 4^{\alpha_4}}{2^{\beta_2} 3^{\beta_3} 5^{\beta_5}}.\]
Consider the map $\Theta_\S:M^{\nu}\rightarrow S^\la$.  If $\dot{\S}$ is the tableau 
\[\dot\S = \rep{1^{\alpha_1} 5^{\alpha_4}}{2^{\beta_2} 3^1 4^{\beta_3-1} 6^{\beta_5}}\]
of type $\dot\nu$ then $\Theta_\S(m_\nu) = \Theta_{\dot{\S}}(m_{\dot{\nu}})$.
Let
\begin{align*}
\mathcal{G} &= \{(g_1,g_4) \mid g_1+g_4=\beta_3, g_1 \leq \alpha_1, g_4 \leq \alpha_4 \}, \\ 
\mathcal{G}'&= \{(g'_1,g'_4) \mid g'_1+g'_4=\beta_3-1, g'_1 \leq \alpha_1, g'_4 \leq \alpha_4 \}.  
\end{align*}
For $g \in \mathcal{G}$, let $\dot\U_g$ be the tableau obtained from $\dot\S$ by moving all entries equal to 3 or 4 from the second row to the first and moving $g_1$ entries equal to 1 and $g_4$ entries equal to 5 from the first to the second.  For $g' \in \mathcal{G'}$, let $\dot\U_{g'}$ be the tableau obtained from $\dot\S$ by moving all entries equal to 4 from the second row to the first and moving $g'_1$ entries equal to 1 and $g'_4$ entries equal to 5 from the first to the second.  Note that $\Theta_{\U_g}(m_{\nu})=\Theta_{\dot\U_g}(m_{\dot\nu})$, where $\U_g$ is as defined in the statement of the lemma.  
Applying the inductive hypothesis repeatedly to $\dot{\S}$, we have 
\begin{align*}
\Theta_\S(m_\nu) &= \rep{1^{\alpha_1} 5^{\alpha_4}}{2^{\beta_2} 3^1 4^{\beta_3-1} 6^{\beta_5}} \\
& = -\q^{-\beta_2} \rep{1^{\alpha_1-1} 3^1 5^{\alpha_4}}{1^1 2^{\beta_2} 4^{\beta_3-1} 6^{\beta_5}} -\q^{\beta_3-1} \rep{1^{\alpha_1} 3^1 5^{\alpha_4}}{2^{\beta_2} 4^{\beta_3-1} 5^1 6^{\beta_5}} \\
&= \sum_{g \in \mathcal{G}} (-1)^{\beta_3}\q^{-\beta_2}\q^{-\binom{\beta_3}{2}+g_1-1}\q^{-(\beta_3-1)(\beta_2+1)}[g_1]\q^{g_4(\beta_2+\beta_3)} \Theta_{\dot\U_g} \\
&\quad + \sum_{g \in \mathcal{G'}} (-1)^{\beta_3} \q^{-\beta_2}\q^{-\binom{\beta_3}{2}+g'_1}\q^{-(\beta_3-1)(\beta_2+1)}[g_1']\q^{\beta_2+1}\q^{g'_4(\beta_2+\beta_3)} \Theta_{\dot\U_{g'}} \\
&\quad + \sum_{g \in \mathcal{G}} (-1)^{\beta_3} \q^{\beta_3-1}\q^{-\binom{\beta_3}{2}+g_1} \q^{-(\beta_3-1)\beta_2} \q^{(g_4-1)(\beta_2+\beta_3-1)}[g_4]\Theta_{\dot\U_g} \\
&\quad + \sum_{g' \in \mathcal{G'}} (-1)^{\beta_3} \q^{\beta_3-1}\q^{-\binom{\beta_3}{2}+g_1'+1}\q^{-(\beta_3-1)\beta_2}\q^{\beta_2}\q^{(g_4'-1)(\beta_2+\beta_3-1)}[g_4']\Theta_{\dot\U_{g'}} \\
&=  \sum_{g \in \mathcal{G}} (-1)^{\beta_3} \q^{-\binom{\beta_3+1}{2}+g_1-\beta_2\beta_3} \q^{g_4(\beta_2+\beta_3)} [\beta_3] \Theta_{\dot\U_g} \\
&\quad + \sum_{g' \in \mathcal{G'}} (-1)^{\beta_3} \q^{-\binom{\beta_3}{2}+g'_1}\q^{g'_4(\beta_2+\beta_3)} \q^{-(\beta_2+1)(\beta_3-1)}\q[\beta_3-1] \Theta_{\dot\U_{g'}}. 
\end{align*}

Applying the inductive hypothesis again, we also have 
\begin{align*}
\Theta_{\dot\S}(m_{\dot\nu}) & = \rep{1^{\alpha_1} 5^{\alpha_4}}{2^{\beta_2} 3^1 4^{\beta_3-1} 6^{\beta_5}} \\
&= \sum_{g' \in \mathcal{G}'}(-1)^{\beta_3-1}\q^{-\binom{\beta_3}{2}+g'_1}\q^{-(\beta_2+1)(\beta_3-1)}\q^{g'_4(\beta_2+\beta_3)} \Theta_{\dot\U_{g'}}(m_{\dot\nu}).
\end{align*}
So, substituting the two values of $\Theta_{\dot\S}(m_{\dot\nu})$ into the equation below, we get that
\begin{align*}
[\beta_3]\Theta_\S(m_\nu) & = \Theta_{\dot\S}(m_{\dot\nu}) + \q[\beta_3-1]\Theta_{\dot\S}(m_{\dot\nu}) \\
&= [\beta_3] \sum_{g \in \mathcal{G}} (-1)^{\beta_3}\q^{-\binom{\beta_3+1}{2}+g_1-\beta_2\beta_3}\q^{g_4(\beta_2+\beta_3)} \Theta_{\dot\U_g}(m_{\dot\nu}) \\
&= [\beta_3] \sum_{g \in \mathcal{G}} (-1)^{\beta_3}\q^{-\binom{\beta_3+1}{2}+g_1-\beta_2\beta_3}\q^{g_4(\beta_2+\beta_3)} \Theta_{\U_g}(m_{\nu}). 
\end{align*}
Since we are working in $R=\calZ$, we may cancel the terms $[\beta_3]$ on both sides of the equation.  This completes the proof of Lemma~\ref{Stage3} when $\S$ has the form $\S=\rep{1^{\alpha_1} 4^{\alpha_4}}{2^{\beta_2} 3^{\beta_3} 5^{\beta_5}}$.  The proof of Lemma~\ref{Stage3} for general $\S$ follows in the same way as the end of the proof of Lemma~\ref{cosetattack}.     
\end{proof}

\begin{lemma} \label{main2part}
Suppose 
\[\S=\rep{1^{\alpha_1} \ldots d^{\alpha_d} \ldots b^{\alpha_b}}{1^{\beta_2} \ldots d^{\beta_d} \ldots b^{\beta_b}}\]
where $\beta_d\geq 1$.  
Let \[\mathcal{G}=\{(g_1,\ldots,g_b) \mid g_d=0, \bar{g}=\beta_d \text{ and } g_i \leq \alpha_i \text{ for } 1 \leq i \leq b\}.\]  For $g \in \mathcal{G}$, let $\U_g$ be the tableau obtained from $\S$ by moving all entries equal to $d$ from row 2 to row 1, and for $i\neq d$ moving down $g_i$ entries equal to $i$ from row 1 to row 2.  Then  
\[\Theta_\S = \sum_{g \in \mathcal{G}} (-1)^{\beta_d} \q^{-\binom{\beta_d+1}{2}+\bar{g}_{d-1}} \q^{-\bar\beta_{d-1}\beta_d}\prod_{i=1}^b \q^{g_i \bar\beta_{i-1}} \gauss{\beta_i+g_i}{g_i}\Theta_{\U_g}.\]
\end{lemma}

\begin{proof}
The case that $\alpha_d=0$ is precisely Lemma~\ref{Stage3}.  So suppose $\alpha_d>0$ and that the lemma holds when $\S^d_{1}<\alpha_d$. 
Let
\begin{align*}
\S(1)&= \rep{1^{\alpha_1}\ldots d^{\alpha_d-1}d+1^1 d+2^{\alpha_{d+1}}\ldots b+1^{\alpha_b}}{1^{\beta_1}\ldots d^{\beta_d} d+2^{\beta_{d+1}}\ldots b+1^{\beta_b}}, \\ 
\S(2) & =  \rep{1^{\alpha_1}\ldots d^{\alpha_d} d+2^{\alpha_{d+1}}\ldots b+1^{\alpha_b}}{1^{\beta_1}\ldots d^{\beta_d-1}d+1^1 d+2^{\beta_{d+1}}\ldots b+1^{\beta_b}}, \\
\end{align*}
and suppose they are of type $\dot\nu$ so that
\[\Theta_\S(m_\nu) = \Theta_{\S(1)}(m_{\dot{\nu}}) + \Theta_{\S(2)}(m_{\dot\nu}).\] Let 
\[\mathcal{G'}=\{(g'_1,\ldots,g'_b) \mid g'_d=0, \bar{g'}=\beta_d -1\text{ and } g'_i \leq \alpha_i \text{ for } 1 \leq i \leq b\}.\]  
For $g \in \mathcal{G}$, let $\dot\U_g$ be the tableau obtained from $\S(1)$ by moving all entries equal to $d$ from row 2 to row 1, and for $i< d$ moving down $g_i$ entries equal to $i$ from row 1 to row 2 and for $i>d+1$ moving $g_{i-1}$ entries equal to $i$ from row 1 to row 2.  
For $g' \in \mathcal{G'}$, let $\dot\U_{g'}$ be the tableau obtained from $\S(2)$ by moving all entries equal to $d$ from row 2 to row 1, and for $i< d$ moving down $g'_i$ entries equal to $i$ from row 1 to row 2 and for $i>d+1$ moving $g_{i-1}$ entries equal to $i$ from row 1 to row 2.  Then, using the inductive hypothesis,
\begin{align*}
\Theta_\S(m_\nu) &= \sum_{g \in \mathcal{G}} (-1)^{\beta_d}\q^{-\binom{\beta_d+1}{2}+\bar{g}_{d-1}}\q^{-\beta_d \bar\beta_{d-1}} \prod_{i=1}^b \q^{g_i \bar\beta_{i-1}}\gauss{\beta_i+g_i}{g_i} \Theta_{\dot\U_g}(m_{\dot\nu})\\ 
&\quad + \sum_{g \in \mathcal{G}} (-1)^{\beta_d} \q^{-\binom{\beta_d+1}{2}+\bar{g}'_{d-1}} \q^{-\beta_d\bar\beta_{d-1}} \q^{\bar{\beta}_d} \prod_{i=1}^b \q^{g_i \bar\beta_{i-1}}\gauss{\beta_i+g_i}{g_i}\Theta_{\dot\U_g}(m_{\dot\nu}) \\
&\quad + \sum_{g' \in \mathcal{G'}} (-1)^{\beta_d-1} \q^{-\binom{\beta_d}{2}+\bar{g}'_{d-1}} \q^{(\beta_d-1)\bar{\beta}_{d-1}} \prod_{i=1}^b \q^{g_i \bar\beta_{i-1}}\gauss{\beta_i+g_i}{g_i} \Theta_{\dot\U_{g'}}(m_{\dot\nu}) \\
&= \sum_{g \in \mathcal{G}} (-1)^{\beta_d} \q^{-\binom{\beta_d+1}{2}+\bar{g}_{d-1}} \q^{-\bar\beta_{d-1}\beta_d}\prod_{i=1}^b \q^{g_i(\bar\beta_{i-1})} \gauss{\beta_i+g_i}{g_i}\Theta_{\U_g}(m_\nu).
\end{align*}   
\end{proof}

We now move on to the more general case where $\la$ may have more than 2 parts.  

\begin{lemma} \label{somemoreparts}
Suppose $\S \in \RowT(\la,\nu)$ where $\la=(\la_1,\ldots,\la_a)$ and $a \geq 2$.  Choose $r$ with $1 \leq r<a$ and suppose that $\S$ satisfies the following conditions: There exists $k$ with $r+1 \leq k$ such that 
\begin{itemize}
\item All entries of $\S$ in rows $1 \leq j <r$ are equal to $j$.
\item All entries of $\S$ in rows $r$ and equal to one of $r,r+1,\ldots,k$ and all entries in row $r+1$ are equal to one of $r+1,r+2,\ldots,k$,  
\item All entries of $\S$ in rows $r+2\leq j \leq a$ are equal to $j+k-r-1$.  
\end{itemize}
Choose $d$ with $\S^d_{r+1} \neq 0$.  
Let \[\mathcal{G}=\{(g_1,\ldots,g_b) \mid g_d=0, \bar{g}=\S^d_{r+1} \text{ and } g_i \leq \S^i_r \text{ for } 1 \leq i \leq b\}.\]  For $g \in \mathcal{G}$, let $\U_g$ be the row-standard tableau obtained from $\S$ by moving all entries equal to $d$ from row $r+1$ to row $r$, and for $i\neq d$ moving down $g_i$ entries equal to $i$ from row $r$ to row $r+1$.  Then  
\[\Theta_\S = \sum_{g \in \mathcal{G}} (-1)^{\S^d_{r+1}} \q^{-\binom{\S_{r+1}^d+1}{2}+\bar{g}_{d-1}} \q^{-\S^{<d}_{r+1}\S^d_{r+1}}\prod_{i=1}^b \q^{g_i \S_{r+1}^{<i}} \gauss{\S^i_{r+1}+g_i}{g_i}\Theta_{\U_g}.\]
\end{lemma}

\begin{proof}
The proof of Lemma~\ref{somemoreparts} is identical to the proof of Lemma~\ref{main2part}, except for the change in notation.  We chose to give the proof of Lemma~\ref{main2part} rather than proving Lemma~\ref{somemoreparts} itself because the notation was easier to control.   
\end{proof}

Now suppose that $\S \in \RowT(\la,\nu)$.  
Choose $r$ with $1 \leq r < a$ and define $\dot\S$ to be the $\la$-tableau such that
\begin{itemize}
\item Each row $1 \leq j <r$ contains $\la_j$ entries equal to $j$.
\item Each row $r+1 <j \leq a$ contains $\la_j$ entries equal to $j+b-2$.
\item Each row $j=r,r+1$ contains $\S^i_j$ entries equal to $i+r-1$, for $1 \leq i \leq b$.  
\end{itemize}

Note that $\dot\S$ satisfies the conditions of Lemma~\ref{somemoreparts}.

\begin{ex}
Suppose that
\[\S=\tab(11122233445,112235,1224,3345),\]
and let $r=2$.  Then
\[\dot\S=\tab(11111111111,223346,2334,7777).\]
\end{ex}

\begin{lemma} \label{LittlePerms}
Suppose $m,x \geq 0$ and $v$ is a permutation of $m+1,\ldots,m+x$.  Suppose $w$ is a permutation such that $w(m+i) < w(m+j)$ for all $1 \leq i<j \leq x$.  Then $\ell(vw) = \ell(v)+\ell(w) = \ell(wv)$.  
\end{lemma}

\begin{proof}
The proof follows by applying Equation~\ref{length}.
\end{proof}

\begin{lemma} \label{NewPerms}
Let $\S \in \RowT(\la,\nu)$.  Choose $r$ with $1 \leq r <a$ and define $\dot\S$ and $\dot\nu$ as above.   
If $w$ is the permutation such that $\first(\S) = \first(\dot\S) w$ then $\ell(d(\first(\S))) = \ell(d(\first(\dot\S)))+ \ell(w)$.
Furthermore if $\T \in \RowT(\la,\nu)$ is such that $\S$ and $\T$ are identical on all rows except possibly rows $r$ and $r+1$ 
then $\first(\T) = \first(\dot(\T)) w$ (so that $\ell(\first(\S)) = \ell(\first(\dot\S)) + \ell(w)$). 
\end{lemma}

\begin{proof}
Note that $d(\first(\dot\S))$ and $w$ satisfy the conditions of Lemma~\ref{LittlePerms} so that $\ell(d(\first(\dot\S))w) = \ell(d(\first(\dot\S)))+ \ell(w)$.  It is straightforward to see that the permutation $w$ works for both $\S$ and $\T$.  
\end{proof}

\begin{lemma} \label{MoreThanThree}
Let $m \geq 0$ and $\eta=(\eta_1,\ldots,\eta_a)$ be a composition such that $a \geq 2$.  Choose $r$ with $1 \leq r <a$.  Then
\[ \CC(m; \eta) = \CC(m+\bar\eta_{r-1}; \eta_r,\eta_{r+1}) \CC(m; \eta_1,\ldots,\eta_{r+1}, \eta_r+\eta_{r+1}) \CC(0; \bar\eta_{r+1},\eta_{r+1},\ldots,\eta_a).\] 
\end{lemma}

\begin{proof}
As usual, we may assume $m=0$.  Applying Lemma~\ref{IntoThree} repeatedly, we get
\begin{align*}
\CC(0;\eta)&= \CC(0;\eta_1,\ldots,\eta_{r+1}) \CC(\bar\eta_{r+1}; \eta_{r+2},\ldots,\eta_{a}) \CC(0; \bar\eta_{r+1},\bar\eta_a-\bar\eta_{r+1}) \\
&= \CC(\bar\eta_{r-1}; \eta_r,\eta_{r+1}) \CC(0; \eta_{1},\ldots,\eta_{r-1}) \CC(0; \bar\eta_{r-1}, \eta_{r}+\eta_{r+1}) \\
&\hspace*{15mm} \CC(\bar\eta_{r+1}; \eta_{r+2},\ldots,\eta_{a}) \CC(0; \bar\eta_{r+1}, \bar\eta_a - \bar\eta_{r+1}) \\
&=\CC(\bar\eta_{r-1}; \eta_r, \eta_{r+1}) \CC(0; \eta_1,\ldots,\eta_{r-1},\eta_r+\eta_{r+1}) \CC(0; \bar\eta_{r+1},\eta_{r+2},\ldots,\eta_{a}).
\end{align*}
\end{proof}

\begin{lemma} \label{Commutes}
Let $m, \eta_1, \eta_2 \geq 0$.  Suppose $w$ is a permutation such that for all $m+1 \leq k \leq m+\eta_1 +\eta_2$ we have $w^{-1}(k) = k+x$ for some $x \in \Z$.  Then
\[ T_w C(m; \eta_1, \eta_2) = C(m+x; \eta_1, \eta_2) T_w.\]
\end{lemma}

\begin{proof}
If $v$ is any permutation of $m+1,\ldots,m+\eta_1+\eta_2$, let $\bar{v}$ be the permutation of $x+m+1,\ldots,x+m+\eta_1+\eta_2$ which sends $x+k$ to $v(k)+x$.  Then clearly $wv = \dot{v}w$ and $\ell(wv) = \ell(w)+\ell(v)$ by Lemma~\ref{LittlePerms}.  The result follows since $C(m;\eta_1,\eta_2)$ is a sum of  basis elements indexed by permutations of $m+1,\ldots,m+\eta_1+\eta_2$.
\end{proof}

\begin{lemma} \label{BadPerm}
Let $\S \in \RowT(\la,\nu)$.  Choose $r$ with $1 \leq r <a$ and define $\dot\S, \, \dot\nu$ and $w$ as in Lemma~\ref{NewPerms}.  Then
\[\Theta_{\S}(m_\nu)=\Theta_{\dot\S}(m_{\dot\nu}) T_{w} \prod_{i=1}^b \CC(\bar\nu_{i-1};\S^i_1,\ldots,\S^i_{r-1},\S^i_r+\S^i_{r+1}) \CC(\bar\nu_{i-1};\S^i_{\leq r+1},\S^{i}_{r+2},\ldots,\S^i_a).\]
\end{lemma}

\begin{proof} Applying Lemmas~\ref{reverse},~\ref{NewPerms},~\ref{MoreThanThree} and~\ref{Commutes}, we have that 
\begin{align*}
\Theta_\S(m_\nu) & = \hla+ m_\la T_{d(\S))} \prod_{i=1}^b \CC(\bar\nu_{i-1};\S^i_1,\ldots,\S^i_a) \\
&= \hla + m_\la T_{d(\first(\dot\S))} T_{w} \prod_{i=1}^b \CC(\bar\nu_{i-1}+\S^i_{ \leq r-1};\S^i_r,\S^i_{r+1}) \CC(\bar\nu_{i-1};\S^i_1,\ldots,\S^i_{r-1},\S^i_r+\S^i_{r+1})\\
& \hspace*{15mm}  \CC(\bar\nu_{i-1}; \S^i_{\leq r+1},\S^{i}_{r+2},\ldots,\S^i_a) \\
&= \hla+ m_\la T_{d(\first(\dot\S))} \prod_{i=1}^b \CC(\bar\la_{i-1}+\S^{<i}_r+\S^{<i}_{r+1};\S^i_r,\S^i_{r+1}) \\
& \hspace*{15mm} T_w \prod_{i=1}^b \CC(\bar\nu_{i-1};\S^i_1,\ldots,\S^i_{r-1},\S^i_r+\S^i_{r+1})  \CC(\bar\nu_{i-1}; \S^i_{\leq r+1},\S^{i}_{r+2},\ldots,\S^i_a)\\ 
& =\Theta_{\dot\S}(m_{\dot\nu}) T_{w} \prod_{i=1}^b \CC(\bar\nu_{i-1};\S^i_1,\ldots,\S^i_{r-1},\S^i_r+\S^i_{r+1}) \CC(\bar\nu_{i-1};\S^i_{\leq r+1},\S^{i}_{r+2},\ldots,\S^i_a).
\end{align*}
\end{proof}

We may now combine the previous results.

\begin{ex}
Let $\S=\tab(1223,1123,123,1)$ so that $\la=(5,4,3,1)$.  Then
\begin{align*}
\tab(1223,1123,123,1) &= \hla+ m_\la T_{(2,6,3,7,8,11,12,5)(4,10,9)} \CC(0;1,2,1,1) \CC(5; 2,1,1) \CC(9;1,1,1) \\
&= \hla+ m_\la  T_{(7,8,10,9)}T_{(2,6,3,7,4,10,11,12,5)} \CC(1;2,1)\CC(7;1,1)\CC(10;1,1) \\
& \hspace*{15mm} \CC(0;1,3)\CC(0;4,1) C(5;2,2) \CC(9;1,2) \\
&= \hla+ m_\la  T_{(7,8,10,9)} \CC(4;2,1)\CC(7;1,1)\CC(9;1,1)\\
&\hspace*{15mm} T_{(2,6,3,7,4,10,11,12,5)} \CC(0;1,3)\CC(0;4,1) C(5;2,2) \CC(9;1,2) \\
&=\tab(1111,2234,234,5) T_{(2,6,3,7,4,10,11,12,5)}  \CC(0;1,3)\CC(0;4,1) C(5;2,2) \CC(9;1,2) \\
&=\Big(-q^{-1}[2]  \tab(1111,2334,224,5)-[2]\tab(1111,2233,244,5) \Big) T_{(2,6,3,7,4,10,11,12,5)} \CC(0;1,3)\CC(0;4,1) C(5;2,2) \CC(9;1,2) \\
&=-q^{-1}[2] \tab(1223,1223,113,1)- [2] \tab(1223,1122,133,1)
\end{align*}
\end{ex}

\begin{proposition} \label{Part1}
Suppose $\la=(\la_1,\ldots,\la_a)$ is a partition of $n$ and $\nu=(\nu_1,\ldots,\nu_b)$ is a composition of $n$.
Let $\S \in \RowT(\la,\nu)$.  
Suppose $1 \leq r\leq a-1$ and that $1 \leq d \leq b$.  Let
\[\mathcal{G} =\left\{g=(g_1,g_2,\ldots,g_b) \mid g_d=0, \, \bar{g}=\S^d_{r+1}, \, \text{ and } g_i \leq \S^{i}_{r} \text{ for } 1 \leq i \leq b\right\}.\]
For $g \in \mathcal{G}$, let $\U_g$ be the row-standard tableau formed by moving all entries equal to $d$ from row $r+1$ to row $r$ and for $i \neq d$, moving $g_i$ entries equal to $i$ from row $r$ to row $r+1$. Then
\[\Theta_\S = \sum_{g \in \mathcal{G}} (-1)^{\S^d_{r+1}} \q^{-\binom{\S^d_{r+1}+1}{2}+\bar{g}_{d-1}} \q^{-S^{<d}_{r+1}\S^d_{r+1}} \prod_{i=1}^b \q^{g_i \S^{<i}_{r+1}} \gauss{\S^i_{r+1}+g_i}{g_i}\Theta_{\U_g}.\]
\end{proposition}

\begin{proof}
Using Lemmas~\ref{somemoreparts},~\ref{NewPerms} and~\ref{BadPerm}, and keeping the notation of Lemma~\ref{BadPerm},
\begin{align*} 
\Theta_\S(m_\nu) &= \Theta_{\dot\S}(m_{\dot\nu}) T_{w} \prod_{i=1}^b \CC(\mu_{i-1};\S^i_1,\ldots,\S^i_{d-1},\S^i_d+\S^i_{d+1}) \CC(\mu_{i-1};\bar\S^i_{d+1},\S^{i}_{d+1},\ldots,\S^i_a) \\
&= \sum_{g \in \mathcal{G}} (-1)^{\S^d_{r+1}} \q^{-\binom{\S_{r+1}^d+1}{2}+\bar{g}_{d-1}} \q^{-\S^{<d}_{r+1}\S^d_{r+1}}\prod_{i=1}^b \q^{g_i(\S_{r+1}^{<i})} \gauss{\S^i_{r+1}+g_i}{g_i}\Theta_{\dot\U_g}(m_{\dot\nu}) \\
&\hspace*{15mm} T_w \prod_{i=1}^b \CC(\mu_{i-1};\S^i_1,\ldots,\S^i_{r-1},\S^i_r+\S^i_{r+1}) \CC(\mu_{i-1};\bar\S^i_{r+1},\S^{i}_{r+1},\ldots,\S^i_a) \\
&= \sum_{g \in \mathcal{G}} (-1)^{\S^d_{r+1}} \q^{-\binom{\S_{r+1}^d+1}{2}+\bar{g}_{d-1}} \q^{-\S^{<d}_{r+1}\S^d_{r+1}}\prod_{i=1}^b \q^{g_i(\S_{r+1}^{<i})} \gauss{\S^i_{r+1}+g_i}{g_i}\Theta_{\U_g}(m_\nu). 
\end{align*}
\end{proof}

The proof of Proposition~\ref{Part1} gives the first half of the proof of Theorem~\ref{Lemma7}.  Since the proof of the second half follows along identical lines, we omit most of it and give only the proof of the analogue of Lemma~\ref{Type1n}, where the difference is non-trivial.    

\begin{ex} Let $\la=(3,3)$ and $\nu=(1,2,1,1,1)$.  
Observe that \[m_\la(I+T_3+T_3T_4+T_3T_4T_5) =(I+T_2+T_1T_2)m_{(2,4)}\in \hla\]
by Lemma~\ref{minihigh}.  Then
\begin{align*}
\tab(134,225) &= \hla+m_{\la}T_3T_4T_2T_3 \\
&= \hla + \q^{-1}m_{\la}T_2T_3T_4T_2T_3 \\
&= \hla + \q^{-1}m_\la T_3T_4T_2T_3T_4 \\
&= \hla - \q^{-1}m_\la(I+T_3+T_3T_4T_5)T_2T_3T_4 \\
&=\hla - m_\la (T_3T_4 + T_3T_2T_4 + \q^{-1}T_3T_4T_5T_2T_3T_4) \\
&= - \tab(124,235) -\q^{-1}\tab(145,223) 
\end{align*}
\end{ex}

\begin{lemma} \label{Type1nb} 
Let $\la=(m,m)$.  
Let $0 \leq \alpha < m$ and $0 \leq \beta \leq  m$ and let $\S \in \rowstd(\la)$ be the tableau with $1,\ldots,\alpha,\alpha+\beta+1,\ldots,m+\beta$ in the top row (and $\alpha+1,\ldots,\alpha+\beta,m+\beta+1,\ldots,2m$ in the second row).  For $i \in \{\alpha+1,\ldots,\alpha+\beta,m+\beta+1,\ldots,2m\}$ let $\U_i$ be the row-standard tableau obtained by swapping $\alpha+\beta+1$ with $i$ and rearranging the rows if necessary.  Then  
\[\Theta_\S = -\sum_{i=\alpha+1}^{\alpha+\beta} \Theta_{\U_i} - \q^{-m+\alpha+1}\sum_{i=m+\beta+1}^{2m}\Theta_{\U_i}.\] 
\end{lemma}

\begin{proof} 
We use Lemma~\ref{PullsThrough} and note that 
\[m_{\la}\CC(m-1;1,m) = \CC(0;m-1,1)m_{(m-1,m+1)} \in \hla\]  
by Lemma~\ref{minihigh}. If $\beta=m$ then
\begin{align*}
\Theta_{\S}(m_\nu) + \sum_{i=m+\beta+1}^{2m} \Theta_{\U_i} 
&= \hla+ m_{\la}\Dt(m,\alpha+1,m) \CC(\alpha;1,m) \\
&= \hla + q^{m-\alpha-1} m_\la \CC(m-1; 1,m) \Dt(m-1,\alpha,m+1) \\
&= \hla \\
&=0.  
\end{align*}
 Else if $\beta<m$ then 
\begin{align*}
\sum_{i=m+\beta+1}^{2m} \Theta_{\U_i}(m_\nu) &= \hla+m_\la T_{m}T_{m+1}\ldots T_{m+\beta}\CC(m+\beta;1,m-\beta-1) 
\Dt(m-1,\alpha,\beta+1) \\
&  = \hla+ m_\la \big(\CC(m-1,1,m)-\CC(m-1;1,\beta)\big) \Dt(m-1,\alpha,\beta+1) \\ 
&=\hla - m_\la \Dt(m-1,\alpha,\beta+1)\CC(\alpha,1,\beta) \\
&= \hla- \q^{m-\alpha-1}  m_\la\Dt(m,\alpha+1,\beta) \CC(\alpha;1,\beta) \\
&= -\q^{m-\alpha-1} \Theta_\S(m_\nu) - \q^{m-\alpha-1} \sum_{i=\alpha+1}^{\alpha+\beta}\Theta_{\U_i}(m_\nu).
\end{align*}
\end{proof}

\subsection{How to prove the results in Section~\ref{Spaces}}\label{SpaceProof}
Let $\h=\h_{\C,q}(\sym_n)$ where $q$ is a primitive $e^\text{th}$ root of unity in $\C$ for some $2 \leq e < \infty$.  In Propositions~\ref{first} to~\ref{last} we described the homomorphism space $\EHom_{\h}(S^\mu,S^\la)$ where $\la$ has at most two parts and $\mu_1 \geq \la_2$.  As previously mentioned, the proof of these results is given by case-by-case analysis, where the calculations consist of solving systems of homogeneous linear equations.  These very many calculations do not belong in a research paper; they are both trivial and lengthy.  We therefore begin this section by giving some identities which  enabled us to solve the equations, followed by the proof of the results in some specific cases.   The  reader who wishes to use one or more of our propositions and who prefers not to rely on results which are not explicitely proved would be best advised to construct their own proofs; we would say this is more tedious than difficult.      

We note that our proof of Proposition~\ref{OneDim}, which states that the homomorphism space is at most 1-dimensional, relies on looking at cases individually; we do not know of a direct proof.  

\begin{lemma} \label{DivZero}
Suppose that $a,b >0$.  Then \[\frac{[ae]}{[be]} = \frac{a}{b}.\]
\end{lemma}

\begin{proof} We have that
\begin{align*}
\frac{[ae]}{[be]} & = \frac{1+q+\ldots +q^{(a-1)e-1}}{1+q+\ldots+q^{(b-1)e-1}} \\
&=\frac{(1+q^e+\ldots+q^{(a-1)e})[e]}{(1+q^e+\ldots+q^{(b-1)e})[e]} \\
&=\frac{a}{b}.  
\end{align*}
\end{proof}

\begin{lemma} \label{WhenZero}
Suppose that $m \geq k \geq 0$.  Write $m=m^\ast e+m'$ and $k=k^\ast e+k'$ where $0 \leq m',k' <e$.  Then $\gauss{m}{k} = 0$ if and only if $m' < k'$.  
\end{lemma}

\begin{proof} We have that
\[\gauss{m}{k}=\frac{[m][m-1]\ldots[m-k+1]}{[k][k-1]\ldots[1]}\]
which, using Lemma~\ref{DivZero} is zero if and only if there are more terms divisible by $e$ in the numerator than the denominator, that is if and only if $m' < k'$.  
\end{proof}

\begin{corollary} \label{e}
Let $m\geq e$.  Then 
\[\gauss{m}{e} \neq 0.\]
\end{corollary}

\begin{corollary}
Let $m,s \geq 1$.  Then 
\[\gauss{m+j}{j} = 0\]
for all $1 \leq j \leq s$
if and only if $e \mid m+1$ and $s<e$.   
\end{corollary}

\begin{lemma}
We have \[(-1)^e q^{\binom{e}{2}}=-1.\]  
\end{lemma}

\begin{proof}
If $e$ is even then \begin{align*} 
(-1)^e q^{\binom{e}{2}} & = q^{\frac{e}{2}(e-1)}=(-1)^{e-1}=-1.\\ \intertext{If it is odd then} (-1)^e q^{\binom{e}{2}} &= - q^{e \frac{e-1}{2}} =-1.
\end{align*}  
\end{proof}

\begin{lemma}[Lemma~\ref{GaussLemma}] \label{Repeat}
Suppose $m, k \geq n \geq 0$.  Then
\[\sum_{\gamma \geq 0}(-1)^\gamma q^{\binom{\gamma}{2}}\gauss{n}{\gamma}\gauss{m-\gamma}{k} = q^{n(m-k)}\gauss{m-n}{k-n}.\]
\end{lemma}

\begin{lemma} \label{NotRepeat}
Suppose $m \geq j \geq 0$ and $n \geq 0$. Then
\[\sum_{\gamma \geq 0} q^{\gamma(m-j+\gamma)} \gauss{n}{\gamma}\gauss{m}{j-\gamma} = \gauss{m+n}{j}.\]
\end{lemma}

\begin{proof}
We use induction on $n$.  The lemma is true for $n=0$ so suppose that $n>0$ and that the lemma holds for $n-1$.  Then using Lemma~\ref{GaussSum},
\begin{align*}
\sum_{\gamma \geq 0} q^{\gamma(m-j+\gamma)} \gauss{n}{\gamma}\gauss{m}{j-\gamma} &= \sum_{\gamma \geq 0} q^{\gamma(m-j+\gamma)} \left( \gauss{n-1}{\gamma-1} + q^\gamma \gauss{n-1}{\gamma} \right) \gauss{m}{j-\gamma} \\
&= \sum_{\gamma \geq 0} q^{(\gamma+1)(m-j+\gamma+1)} \gauss{n-1}{\gamma}\gauss{m}{j-\gamma-1} + \sum_{\gamma \geq 0} q^{\gamma(m-j+\gamma)} q^\gamma \gauss{n-1}{\gamma}\gauss{m}{j-\gamma} \\
&= \sum_{\gamma \geq 0} q^{\gamma(m-j+\gamma)}\gauss{n-1}{\gamma}\left(q^{m-j+\gamma+1} \gauss{m}{j-\gamma-1} + \gauss{m}{j-\gamma} \right) \\
&= \sum_{\gamma \geq 0} q^{\gamma(m-j+\gamma+1)} \gauss{n-1}{\gamma} \gauss{m+1}{j-\gamma} \\
&= \gauss{m+n}{j} 
\end{align*}
by the inductive hypothesis.  
\end{proof}

Armed with these results, we are ready to start solving some equations.  Let us begin with the simplest non-trivial case which is when $\mu$ has three parts.

Take $\la=(\la_1,\la_2)$ and $\mu=(\mu_1,\mu_2,\mu_3)$ to be partitions of $n$ where $\la_1 \geq \mu_1 \geq \la_2$.  Our aim is to determine $\Hom_\h(S^\mu,S^\la)$ by finding a basis for $\Psi(\mu,\la)$.  
For $\max\{0,\la_2-\mu_3\} \leq i \leq \min\{\mu_2,\la_2\}$ let $\Theta_i:M^\mu\rightarrow S^\la$ be given by
\[\Theta_i(m_\mu) = \rep{1^{\mu_1}2^{\mu_2-i}3^{\mu_3-\la_2+i}}{2^i3^{\la_2-i}},\]
where we use the notation of Section~\ref{Lemma7Proof}.  
Then $\Psi(\mu,\la)$ is the vector space of homomorphisms
\[\Theta=\sum_i a_i \Theta_i,\]
for $a_i \in \C$, which satisfy $\Theta(m_\mu h_{d,t})=0$ for $d=1,2$ and $1 \leq t \leq \mu_{d+1}$.  Note that for such $d,t$, we have
\begin{align*}
\Theta_i(m_\mu h_{2,t}) & = \sum_j q^{i(t-j+i)}  \gauss{\mu_2+t-j}{\mu_2-i} \gauss{j}{i} \rep{1^{\mu_1}2^{\mu_2+t-j}3^{\mu_3-t-\la_2+j}}{2^j3^{\la_2-j}}, \\
\Theta_i(m_\mu h_{1,t}) & = \sum_k  (-1)^{k-i}q^{\binom{i-k}{2} +kt} \gauss{\mu_1+t-i}{\mu_1-k} \gauss{\la_2-k}{\la_2-i}\rep{1^{\mu_1+t}2^{\mu_2-t-k}3^{\mu_3-\la_2+k}}{2^k3^{\la_2-k}},
\end{align*}
where the sums are over all $j$ such that $\max\{0,\la_2-\mu_3+t\} \leq j \leq \min\{\la_2,\mu_2+t\}$ and all $k$ such that $\max\{0,\la_2-\mu_3\} \leq k \leq \min\{\la_2,\mu_2-t\}$. 
We must therefore solve the systems of equations 
\begin{align*}
\sum_i q^{i(t-j+i)}  \gauss{\mu_2+t-j}{\mu_2-i} \gauss{j}{i}  a_i &=0, &&(2,t,j) \\
\sum_i  (-1)^{k-i}q^{\binom{i-k}{2} +kt} \gauss{\mu_1+t-i}{\mu_1-k} \gauss{\la_2-k}{\la_2-i}a_i & = 0, && (1,t,k)
\end{align*}
for $i,t,j,k$ within the limits described above.  Let us redefine $\Psi=\Psi(\mu,\la)$ to be the solution space of the equations $(1,t,k)$ and $(2,t,j)$ and $\Phi=\Phi(\mu,\la)$ to be the solution space of the equations $(2,t,j)$. 

Finally, for any $m \geq 0$, we define $m^\ast$ and $m'$ by writing $m=m^\ast e + m'$ where $0 \leq m' <e$.  The equivalence relation $\equiv$ is assumed to be equivalence modulo $e$.  

\begin{lemma} \label{GoodLemma}
Suppose $0 \leq M \leq N$ and that $(a_i)_{M \leq i \leq N}$ is a non-zero solution to the system of equations
\begin{align*}[j+1]a_j +q^{j+1}[K-j]a_{j+1} &=0, && M \leq j \leq N-1,
\end{align*} 
where $N \leq K+1$.  
If there do not exist either $c$ with $M+1 \leq c \leq N$ and $c \equiv 0$ or $d$ with $M+1 \leq d \leq N$ and $K-d+1 \equiv 0$ then $a_i \neq 0$ for all $M\leq i \leq N$.  Otherwise $a_i \neq 0$ only if $i' \geq (K+1)'$ and furthermore if $M \leq i,k \leq N$ are such that $i', k' \geq (K+1)'$ and $i^\ast = k^\ast$ then $a_i \neq 0$ if and only if $a_k \neq 0$.  
\end{lemma}

\begin{proof}
If neither $c$ nor $d$ as above exist then all terms $[M+1],\ldots,[N],[K-M],\ldots,[K-N+1]$ are non-zero and the result is clear.  So suppose at least one of $c,d$ exists.  Let $b=(K+1)'$ and suppose $M \leq i \leq N$ is such that $i'<b$.  If $i^\ast e-1 \geq M$ then the equations 
\begin{align*}[j+1]a_j +q^{j+1}[K-j]a_{j+1} &=0, && i^\ast e-1 \leq j \leq i-1,
\end{align*} 
ensure that $a_i=0$.  Similarly if $i^\ast e + b \leq N$ then the equations 
\begin{align*}[j+1]a_j +q^{j+1}[K-j]a_{j+1} &=0, && i \leq j \leq i^\ast e+b-1,
\end{align*} 
ensure that $a_i=0$.  But if $M \geq i^\ast e$ and $N<i^\ast e+b$ then we cannot find $c$ or $d$ as above.    Hence if $i'<b$ then $a_i=0$.  

Now suppose $M \leq i,k \leq N$ are such that $i', k' \geq b$ and $i^\ast = k^\ast$.  Assume $i \leq k$.  The equations
\begin{align*}[j+1]a_j +q^{j+1}[K-j]a_{j+1} &=0, && i \leq j \leq k-1,
\end{align*} 
ensure that $a_i=0$ if and only if $a_k=0$. 
\end{proof}

\begin{lemma} \label{lemmalow}
Suppose $\la_2 < \mu_3$.  Then
\[\dim(\Phi(\mu,\la)) \leq \begin{cases} 0, & e\leq \la_2, \\
1, & (\mu_2+1)' \leq \la_2 < e, \\
0, & \la_2 < (\mu_2+1)'.
\end{cases}\]
\end{lemma}

\begin{proof}
Suppose $(a_i)_{0 \leq i \leq \la_2} \in \Phi$.    
Setting $t=1$, we have the equations
\begin{align*}
[\mu_2+1] a_0 & = 0, &&\\
[j+1]a_j +q^{j+1}[\mu_2-j]a_{j+1} &=0, && 0 \leq j \leq \la_2-1.
\end{align*}
Let $b=(\mu_2+1)'$.  If $\la_2 <b$ then there can be no non-zero solution to these equations. 
Otherwise, applying Lemma~\ref{GoodLemma}, $a_i \neq 0$ only if $b \leq i'$ and in this case $a_i \neq 0$ if and only if $a_{i^\ast e +b} \neq 0$ so that if $\la_2 < e+b$ then $\dim(\Phi) \leq 1$.     

Suppose $\la_2 \geq e+b$ and consider the equation $(2,e,(r+1)e+b)$ where $0 \leq re+b \leq \la_2-e$.  From this, and using Lemma~\ref{WhenZero},
\[0=\sum_{i=re+b}^{(r+1)e+b} q^{i(i-re-b)}\gauss{\mu_2-re-b}{\mu_2-i}\gauss{(r+1)e+b}{i} a_i = \gauss{(r+1)e+b}{e} a_{re+b} + \gauss{\mu_2-re-b}{e}a_{(r+1)e+b},\]
so that, applying Lemma~\ref{e}, $a_{re+b} \neq 0$ if and only if $a_{(r+1)e+b} \neq 0$.  Hence if $\la_2 \geq e$ then $\dim(\Phi) \leq 1$ and if $(a_i) \in \Phi \setminus \{0\}$ then $a_{e-1} \neq 0$.   But in this case we obtain a contradiction by considering Equation $(2,e,e-1)$ and using Lemma~\ref{e}: 
\[0= \sum_{i=0}^{e-1} q^{i(e-1-i)} \gauss{e-1}{i} \gauss{\mu_2+1}{\mu_2-i}a_i = q^{\la_2 e} \gauss{\mu_2+1}{e} a_{e-1}\]
since $a_i = 0$ for $0 \leq i <b$ and the second Gaussian polynomial is zero for $b \leq i < e-1$.   So $\dim(\Phi)=0$ if $\la_2 \geq e$.
\end{proof}

\begin{proposition}
Suppose that $\la_2 < \mu_3$.  Then 
\[\dim(\Psi(\mu,\la)) = \begin{cases}
1, & \la_2<e, \, \mu_2=e-1 \text{ and } \mu_1-\la_2+1 \equiv 0, \\
0, & \text{otherwise.}
\end{cases}\]
\end{proposition}

\begin{proof}  
Let $b=(\mu_2+1)'$.  
By Lemma~\ref{lemmalow}, we have $\dim(\Psi) \leq 1$ and if $\dim(\Psi)>0$ then $b \leq \la_2 <e$; furthermore if $(a_i)_{0 \leq i \leq \la_2} \in \Psi \setminus \{0\}$ then $a_i \neq 0$ if and only if $b \leq i \leq \la_2$.  Then if $(a_i) \in \Psi\setminus\{0\}$ then it also satisfies 
\begin{align*}
[\mu_1+1-k]a_k & = [\la_2-k]a_{k+1}, && 0 \leq k < \la_2, \\
[\mu_1+1-\la_2]a_{\la_2} &=0. &&
\end{align*}
Therefore if $\dim(\Psi)>0$ then $\mu_1-\la_2+1 \equiv 0$, and so $a_{\la_2-1}=a_{\la_2}$ so that by Equation $(2,1,\la_2-1)$ we have $\mu_2+1 \equiv 0$ and $b=0$.  So if $\dim(\Psi) >0$, then $\Psi$ is spanned by $(a_i)$ where $a_i=1$ for all $i$.  

Suppose $b \leq \la_2<e$ and that $\mu_1-\la_2+1 \equiv 0$.  
Set $a_i =1$ for all $i$.  Now if $\mu_2 \geq e$ then by using Lemma~\ref{e} and considering Equation $(1,e,0)$ and Lemma~\ref{Repeat} we obtain the contradition 
\[0=\sum_{i=0}^{\la_2} (-1)^i q^{\binom{i}{2}}\gauss{\la_2}{i} \gauss{\mu_1+e-i}{\mu_1}a_i = q^{\la_2 e} \gauss{\mu_1-\la_2+e}{\mu_1-\la_2},\]
so we must have if $\dim(\Psi)>0$ then $\mu_2 = e-1$.  
 
We have shown that if $\dim(\Psi) =1$ then $\la_2<e$, $\mu_2=e-1$ and $\mu_1-\la_2+1 \equiv 0$ and the space $\Psi$ is spanned by the solution $(a_i)$ where $a_i=1$ for $0 \leq i \leq \la_2$.   We now show these conditions are sufficient.  
First we must show that the equations $(2,t,j)$ where $1 \leq t \leq \mu_3$ and $\max\{0,\la_2-\mu_3+t\} \leq j \leq \la_2$ are satisfied by the solution $a_i=1$ for all $i$.  Using Lemma~\ref{NotRepeat} we have
\[\sum_i q^{i(t-j+i)} \gauss{\mu_2+t-j}{\mu_2-i} \gauss{j}{i} = \gauss{\mu_2+t}{t} =0\]
by Lemma~\ref{WhenZero}, since $\mu_2=e-1$, so these equations do hold.  Now consider the equations $(1,t,k)$ where $1 \leq t \leq \mu_2=e-1$ and $0 \leq k \leq \min\{\la_2,\mu_2-t\}$.  By Lemma~\ref{Repeat} we have
\begin{align*}
\sum_{i} (-1)^{k-i} q^{\binom{i-k}{2}} \gauss{\mu_1+t-i}{\mu_1-k} \gauss{\la_2-k}{\la_2-i}  
&= \sum_{i} (-1)^i q^{\binom{i}{2}} \gauss{\la_2-k}{i} \gauss{\mu_1+t-k-i}{\mu_1-k} \\
& = q^{(\la_2-k)t} \gauss{\mu_1-\la_2+t}{t} \\
& =0
\end{align*}
again by Lemma~\ref{WhenZero}, since $\mu_1-\la_2 \equiv -1$, so these equations are also satisfied.    
\end{proof}

We believe this proof should convince the reader that solving all the equations required for Propositions~\ref{first} to~\ref{last} is a lengthy business; but we hope that we have also convinced them that it is not particularly difficult.  A strategy that works is as follows: Write down a system of equations $(d,t,j)$ which need to be solved, then use the equations $(d,1,j)$ and $(d,e,j)$ to come up with necessary conditions for the space $\Psi(\mu,\la)$ to be non-zero.  These conditions turn out to be sufficient.    

\section{Dipper-James Specht modules} \label{DJSpecht}
We conclude by making a connection with the Specht modules of Dipper and James.  
Recall that for each partition $\la$ of $n$, Dipper and James~\cite[Section~4]{DJ:Hecke} defined a $\h$-module which they called a Specht module and which we shall denote $S(\la)$.  The connection with the modules $S^\la$ is that 
\begin{align*}
S(\la) \cong_\h (S^{\la})^\diamond&&&(\ddag)\end{align*}
where $M^\diamond$ denotes the dual of a $\h$-module $M$~\cite[Theorem~5.3]{Murphy}.  

\begin{proposition}
We have the isomorphisms
\begin{align*}
\Hom_{\h}(S(\la),S(\mu)) &\cong_F \Hom_\h(S^{\mu},S^{\la}) \\
&\cong_F\Hom_\h(S^{\la'},S^{\mu'}) \\
&\cong_F\Hom_\h(S(\mu'),S(\la')).
\end{align*}
\end{proposition}

\begin{proof}
The first and last equations follow from Equation $(\ddag)$ above.  Now, following~\cite[Lemma~2.3]{Murphy}, we define a $\h$-automorphism $\#$.  For a $\h$-module $M$, define $M^\#$ to be the right $\h$-module with action defined by 
$m \cdot h = mh^\#$ for all $m \in M$ and $h\in \h$.  Then $(S^{\nu'})^\# = (S^\nu)^\diamond$ for any partition $\nu$ of $n$~\cite[Theorem~5.2]{Murphy}, and the middle equation follows.    
\end{proof}

If $q=-1$ then we might also wish to replace $\Hom_\h(\, , \,)$ with $\EHom_\h(\, , \,)$ above.  We now use Theorem~\ref{Weyl}.  The middle equation follows since 
\[\Hom_{\mathcal{S}}(\Delta^\mu,\Delta^\la) \cong\Hom_{\mathcal{S}}(\Delta^{\la'},\Delta^{\mu'});\]
see for example \cite[Lemma~3.4]{LM:rowhoms}.   Recall that for each partition $\la$ of $n$, Dipper and James~\cite{DJ:qWeyl} defined a $\mathcal{S}$-module which they called a Weyl module and which we shall denote $\Delta(\la)$.  The relationship with the Weyl modules of Theorem~\ref{Weyl} is that $\Delta(\la) \cong_{\mathcal{S}} \Delta^{\la'}$.  Futhermore 
\[\Hom_{\mathcal{S}}(\Delta(\la),\Delta(\mu))\cong \EHom_{\h}(S(\la),S(\mu))\]
and so
\[\EHom_{\h}(S^{\la'},S^{\mu'})\cong \EHom_{\h}(S(\la),S(\mu))\]
as required.  

Since many authors prefer to work in the Dipper-James world, it is worth remarking, once and for all, that all the relevant combinatorics can be translated backwards and forwards without change.  
In particular, there exists an analogue of Theorem~\ref{Lemma7}.  Suppose $\la$ is a partition of $n$ and $\nu$ a composition of $n$. Recall that for each $\S \in \RowT(\la,\nu)$ Dipper and James~\cite{DJ:Hecke} defined a homomorphism from $S^\la$ to $M^\nu$ which we shall now denote $\varphi_{\S}$.  

\begin{theorem} \label{DJ:Lemma7}
Suppose $\la=(\la_1,\ldots,\la_a)$ is a partition of $n$ and $\nu=(\nu_1,\ldots,\nu_b)$ is a composition of $n$.
Let $\S \in \RowT(\la,\nu)$.  
\begin{enumerate}
\item 
Suppose $1 \leq r\leq a-1$ and that $1 \leq d \leq b$.  Let
\[\mathcal{G} =\left\{g=(g_1,g_2,\ldots,g_b) \mid g_d=0, \, \bar{g}=\S^d_{r+1} \text{ and } g_i \leq \S^{i}_{r} \text{ for } 1 \leq i \leq b\right\}.\]
For $g \in \mathcal{G}$, let $\U_g$ be the row-standard tableau formed by moving all entries equal to $d$ from row $r+1$ to row $r$ and for $i \neq d$ moving $g_i$ entries equal to $i$ from row $r$ to row $r+1$. Then
\[\varphi_\S = (-1)^{\S^d_{r+1}} q^{-\binom{\S^d_{r+1}+1}{2}} q^{-\S^d_{r+1}S^{<d}_{r+1}} \sum_{g \in \mathcal{G}} q^{\bar{g}_{d-1}} \prod_{i=1}^b q^{g_i \S^{<i}_{r+1}} \gauss{\S^i_{r+1}+g_i}{g_i}\varphi_{\U_g}.\]
\item 
Suppose $1 \leq r\leq a-1$ and $\la_r=\la_{r+1}$ and that $1 \leq d \leq b$.  Let
\[\mathcal{G} =\left\{g=(g_1,g_2,\ldots,g_b) \mid g_d=0, \, \bar{g} = \S^d_r \text{ and } g_i \leq \S^{i}_{r+1} \text{ for } 1 \leq i \leq b \right\}.\]
For $g \in \mathcal{G}$, let $\U_g$ be the row-standard tableau formed by moving all entries equal to $d$ from row $r$ to row $r+1$ of $\S$ and for $i \neq d$ moving $g_i$ entries equal to $i$ from row $r+1$ to row $r$. Then
\[\varphi_\S =  (-1)^{\S^d_{r}} q^{-\binom{\S^d_{r}}{2}} q^{-\S^d_r \S^{>d}_r} \sum_{g \in \mathcal{G}} q^{-\bar{g}_{d-1}}  \prod_{i=1}^b q^{g_i \S^{>i}_{r}} \gauss{\S^i_{r}+g_i}{g_i} \varphi_{\U_g}.\]
\end{enumerate}
\end{theorem}

The proof of Theorem~\ref{DJ:Lemma7} involves working through the same steps as the proof of Theorem~\ref{Lemma7}.  Rather than working modulo $\hla$, we now use the fact that $S(\la)$ is in the kernel of any homomorphism $M(\la) \rightarrow S(\sigma)$ where $\sigma \rhd \la$.  We leave the details to the reader.  

\begin{ex}
Let $\la=(3,3)$ and $\nu=(2,1,1,1,1)$ (as in Example~\ref{NiceEx}). 
Recall that $S(\la)$ is generated by $m_\la X$ for a certain $X \in \h$ (see~\cite{DJ:Hecke} for details) and that, by the kernel intersection theorem,
\[0=m_{(4,2)}(I+T_4+T_4T_5)X=(1+T_3+T_2T_3+T_1T_2T_3)m_{(3,3)}X.\]
Let $\S=\tab(114,235)$.  Then
\begin{align*}
\varphi_{\S}(m_\la X) &= m_{\nu} T_4 T_3 (I+T_2+T_2T_1)(I+T_4)(I+T_5+T_5T_4)X \\
&=T_4 T_3 m_{\la} X \\
&=\q^{-1} T_4T_3T_4 m_{\la}X \\
&= \q^{-1} T_3T_4T_3 m_\la X \\
&= - \q^{-1} T_3T_4(I+T_2T_3+T_1T_2T_3)m_\la X \\
&=- (T_3 +\q^{-1}T_3T_2T_4T_3 + \q^{-1}T_3T_4T_1T_2T_3)m_\la X \\
&= - m_\nu T_3 (I+T_2+T_2T_1)(I+T_4)(I+T_5+T_5T_4)X \\
& \hspace*{15mm} -\q^{-1} m_\nu T_3T_2 T_4T_3 (I+T_4)(I+T_5+T_5T_4)(I+T_1)(I+T_2+T_2T_1)X \\
&= - \varphi_{\U_1}(m_\la X) -\q^{-1} \varphi_{\U_2}(m_\la X) 
\end{align*}
where
\[\U_1=\tab(113,245), \qquad \qquad \U_2=\tab(134,125).\]
\end{ex}

\end{document}